\documentclass[12pt,leqno]{article}
\usepackage[english,activeacute]{babel}
\usepackage{epsfig}
\usepackage{color}

\usepackage[latin1]{inputenc}
\usepackage[T1]{fontenc}
\usepackage{amsmath,amsfonts,amssymb,mathrsfs,amsthm}
\usepackage{graphicx}
\usepackage[mathscr]{eucal}
\usepackage{makeidx}
\usepackage{verbatim}
\usepackage{graphics,graphicx}
\usepackage{textcomp}

\newcommand\R{\mathbb{R}}
\newcommand\C{\mathbb{C}}
\newcommand\N{\mathbb{N}}
\newcommand\Z{\mathbb{Z}}

\newcommand\bna{\begin{eqnarray*}}
\newcommand\ena{\end{eqnarray*}}

\newcommand\bnan{\begin{eqnarray}}
\newcommand\enan{\end{eqnarray}}

\newcommand\Hu{H^1(M)}

\newcommand\Xsbt{X^{s,b}_{T}}
\newcommand\Xsb{X^{s,b}}
\newcommand\Xubt{X^{1,b}_{T}}

\newcommand\intt{\int_0^t}
\newcommand\intT{\int_0^T}
\newcommand\iintS{\iint_{\Sigma}}

\newcommand\Dem{\textbf{Proof : }}

\newcommand\nor[2]{\left\|#1\right\|_{#2}}
\newcommand\norL[1]{\left\|#1\right\|_{L^2}}

\newcommand\troisp{\cdot\cdot\cdot}

\newcommand\tend[2]{\underset{#1 \to #2}{\longrightarrow}}

\newcommand\Tu{\mathbb{T}^1}
\newcommand\Tot{\mathbb{T}^3}

\newcommand\n{\vec{n}}
\newcommand\ubar{\overline{u}}

\newtheorem{theorem}{Theorem}[section]
\newtheorem{remarque}{Remark}[section]
\newtheorem{lemme}{Lemma}[section]
\newtheorem{corollaire}{Corollary}[section]
\newtheorem{prop}{Proposition}[section]
\newtheorem{hypo}{Assumption}

\title{Global controllability and stabilization for the nonlinear Schrödinger equation on some compact manifolds of dimension 3}
\author{Camille Laurent \thanks{Universit\'e Paris-Sud, Laboratoire de Math\'{e}matiques d'Orsay, Orsay Cedex, F-91405 (camille.laurent@math.u-psud.fr).}}
\begin{document}

\maketitle
\begin{abstract}
We prove global internal controllability in large time for the nonlinear Schr\"odinger equation on some compact manifolds of dimension $3$. The result is proved under some geometrical assumptions : geometric control and unique continuation. We give some examples where they are fulfilled on $\Tot$, $S^3$ and $S^2\times S^1$.
We prove this by two different methods both inherently interesting. The first one combines stabilization and local controllability near $0$. The second one uses successive controls near some trajectories. We also get a regularity result about the control if the data are assumed smoother. If the $H^1$ norm is bounded, it gives a local control in $H^1$ with a smallness assumption only in $L^2$. We use Bourgain spaces. 
\end{abstract}

{\bf Key words.} Controllability, Stabilization, Nonlinear Schr\"odinger equation, Bourgain spaces

\vspace{0.2 cm}{\bf AMS subject classifications.} 93B05, 93D15, 35Q55, 35A21 
\tableofcontents
\section*{Introduction}

In this article, we study the internal stabilization and exact controllability for the defocusing nonlinear Schrödinger equation (NLS) on some compact manifolds of dimension $3$. 
\begin{eqnarray}
\label{eqncontrolnl}
\left\lbrace
\begin{array}{rcl}
i\partial_t u + \Delta u &=&|u|^2u \quad \textnormal{on}\quad [0,+\infty [\times M\\
u(0)&=&u_{0} \in H^1(M).
\end{array}
\right.
\end{eqnarray}
where $\Delta$ is the Laplace-Beltrami operator on $M$. The solution displays two conserved energy : the $L^2$ energy $\norL{u}$ and the nonlinear energy, or $H^1$ energy
\bna
E(t)=\frac{1}{2}\int_M \left|\nabla u \right|^2 +\frac{1}{4}\int_M \left|u\right|^4.
\ena
Some similar results where obtained in dimension $2$ in the article of B. Dehman, P. Gérard and G. Lebeau \cite{control-nl} where exact controllability in $H^1$ is proved  for defocusing NLS on compact surfaces. Yet, the proof is based on Strichartz estimates which provide uniform wellposedness in dimension $3$, only in $H^s$ for $s>1$. In \cite{Strichartz}, N. Burq, P. Gérard and N. Tzvetkov managed to prove global existence and uniqueness in $H^1$ but failed to prove uniform wellposedness, which appears of great importance in control problems. Yet, for certain specific manifolds, the strategy of $\Xsb$ spaces of J. Bourgain, extended to some other manifolds by  Burq, Gérard and Tzvetkov, succeeded in proving uniform wellposedness in $H^s$ for some lower regularities. 

For control results, the $\Xsb$ spaces have already been used in dimension $1$ at $L^2$ regularity : first L. Rosier and B. Y. Zhang \cite{RosierZhang} obtained local results and independently, we proved global controllability in large time in \cite{LaurentNLSdim1}. The $\Xsb$ spaces will also be our framework in this paper. 

Under some specific assumptions that will be precised later, we prove global controllability in large time by two different ways, interesting for their own : by stabilization and control near $0$ or by some successive controls near some trajectories. This will provide global controllability towards $0$, the general result will follow by reversing time.

The assumptions are fulfilled in the following cases ($\omega \subset M$ is the support of the control) : \\
- $\Tot$ with $\omega=\left\{x\in \R^3/(\theta_1 \Z\times \theta_2 \Z\times \theta_3) \Z \left|\exists i\in \{1,2,3\}, x_i\in ]-\varepsilon,\varepsilon[+\theta_i\Z \right.\right\} $ (that is a neighborhood of each face of the "cube", fundamental volume of $\Tot$) with $\theta_i \in \R$. Moreover, we can easily extend this result to a cuboid with Dirichlet or Neumann boundary conditions, see  \cite{LaurentNLSdim1} or \cite{RosierZhang}. \\
- $S^3$ with $\omega$ a neighborhood of $\left\{x_4=0\right\}$ in $S^3\subset \R^4$.\\
-$S^2\times S^1$ with $\omega= (\omega_1 \times S^1 )\cup (S^2 \times ]0,\varepsilon[)$ where $\omega_1$ is a neighborhood of the equator of $S^2$.
\begin{theorem}
\label{thmcontrol}
For any open set $\omega \subset M$ satisfying Assumption \ref{geometriccontrol}, \ref{uniquecontinuation}, \ref{hypinegXsb} (see below) and $R_0>0$, there exist $T>0$ and $C>0$ such that for every $u_0$ and $u_1$ in $H^1(M)$ with
$$ \nor{u_0}{H^1(M)} \leq R_0 \quad \textnormal{and}\quad \nor{u_1}{H^1(M)} \leq R_0$$
there exists a control $g\in C([0,T],H^1)$ with $\nor{g}{L^{\infty}([0,T],H^1)}\leq C $ supported in $[0,T]\times \overline{\omega}$, such that the unique solution $u$ in $X^{1,b}_T$ of the Cauchy problem
\begin{eqnarray}
\label{eqncontrolintro}
\left\lbrace
\begin{array}{rcl}
i\partial_t u + \Delta u &=&|u|^2u +g\quad \textnormal{on}\quad[0,T]\times M\\
u(0)&=&u_{0} \in H^1(M)
\end{array}
\right.
\end{eqnarray}
satisfies $u(T)=u_1$.
\end{theorem}
In all the rest of the article, $\omega$ will be related to a cut-off function $a=a(x)\in C^{\infty}(M)$ (whose existence is guaranteed by Whitney Theorem), taking real values and such that 
\bnan
\label{lienoma}
\omega = \left\{x\in M : a(x)\neq 0 \right\}.
\enan
The stabilization system we consider is 
\begin{eqnarray}
\label{eqndampedL2intro}
\left\lbrace
\begin{array}{rcl}
i\partial_t u + \Delta u -a(x)(1-\Delta )^{-1}a(x)\partial_t u&=&(1+|u|^2)u \quad \textnormal{on}\quad[0,T]\times M\\
u(0)&=&u_{0} \in H^1(M).
\end{array}
\right.
\end{eqnarray}
The link with the original equation can be made by the change of variable $w=e^{-it}u$. The well posedness of this system will be proved in Section \ref{sectexistence} and we can check that it satisfies the energy decay
\bnan
E(u(t))-E(u(0))=-\intt \left\|(1-\Delta)^{-1/2}a(x)\partial_tu\right\|_{L^2}^2.
\enan
Our theorem states that under some geometrical hypotheses, this yields an exponential decay.
\begin{theorem}
\label{thmstab}
Let $M$, $\omega$ satisfying Assumption \ref{geometriccontrol}, \ref{uniquecontinuation}, \ref{hypinegXsb}. Let $a \in C^{\infty}(M)$, as in  (\ref{lienoma}). There exists $\gamma>0$ such that for every $R_0>0$, there is a constant $C>0$ such that inequality
$$\nor{u(t)}{H^1} \leq Ce^{-\gamma t} \nor{u_0}{H^1} \quad t>0$$
holds for every solution $u$ of system (\ref{eqndampedL2intro}) with initial data $u_0$ such that $\left\|u_0\right\|_{H^1}\leq R_0$.
\end{theorem}
The independence of $C$ and of the time of control $T$ on the bound $R_0$ are an open problem. The fact that $\gamma$ is independant on the size lies on the fact that it only describes the behavior near $0$. However, it is unknown whether there is really a minimal time of controllability. This is in strong contrast with the linear case where exact controllability occurs in arbitrary small time and the conditions are only geometric for the open set $\omega$. Moreover, some recent studies have analysed the explosion of the control cost when $T$ tends to $0$ : K.- D. Phung \cite{Phung} by reducing to the heat or wave equation, L. Miller \cite{Miller} with resolvent estimates, G. Tenenbaum and M. Tucsnak \cite{Tenenbaum} with number theoretic arguments.

Let us now describe our assumptions. The first two deal with classical geometrical assumptions in control theory.
\begin{hypo}
\label{geometriccontrol}Geometric control : there exists $T_0>0$ such that every geodesic of $M$, travelling with speed $1$ and issued at $t=0$, enters the set $\omega$ in a time $t<T_0$.
\end{hypo}
This condition is known to be sufficient for linear controllability, see G. Lebeau \cite{control-lin1}. In Section \ref{sectGCCnec}, we prove that it is necessary on $S^3$ for the nonlinear stabilization. Yet, there are some geometrical situation (especially when there are some unstable geodesics) in which it is  not necessary. For example, we have linear controllability for any open set $\omega$ of $\mathbb{T}^3$, see S. Jaffard \cite{Jaffard} and V. Komornik \cite{Komornik} (see also \cite{Burq}). This also holds for $M=S^2\times S^1$ with $\omega=S^2\times ]0,\varepsilon[$ or $\omega=\omega_1 \times S^1 $ where $\omega_1$ is a neighborhood of the equator. In that cases, our method fails to prove global results and we can only prove local controllability by perturbation (see Theorem \ref{thmcontrolenonlinw0}).
\begin{hypo}
\label{uniquecontinuation}
Unique continuation : For every $T>0$, the only solution in $C^{\infty}([0,T]\times M)$ to the system 
\begin{eqnarray}
\left\lbrace
\begin{array}{c}
i\partial_t u + \Delta u + b_1(t,x)u+b_2(t,x)\overline{u}=0  \textnormal{ on } [0,T]\times M\\
u=0 \textnormal{ on } [0,T]\times \omega
\end{array}
\right.
\end{eqnarray}
where $b_1(t,x)$ and $b_2(t,x) \in C^{\infty}([0,T]\times M)$ is the trivial one $u \equiv 0$.
\end{hypo}
We do not know if there exists a link between these two assumptions. In our three particular cases, this can be proved using Carleman estimates. There are some existing results about this, as the one of V. Isakov \cite{Isakov}(for general anisotropic PDE's), L. Baudouin and J.P. Puel \cite{BaudouinPuel2002}(global Carleman estimates) or A. Mercado, A. Osses and L. Rosier\cite{CarlemanSchrodRosier}(in the special case of Schrödinger with flat metric but weaker geometrical assumptions). Then, for the convenience of the reader, we have chosen to give a proof of this, which, we believe, clarifies the problem in the case of a non flat metric. It is given in the Appendix, Section \ref{sectuniqueness}.

The last assumption is a technical assumption that ensures that the Cauchy problem is well posed in $H^1$. It yields a bilinear loss of $s_0<1$.
\begin{hypo}
\label{hypinegXsb}
There exists $C>0$ and $0\leq s_0<1$ such that for any $f_1,f_2\in L^2(M)$ satisfying
\bna
f_j=\mathbf{1}_{\sqrt{1-\Delta}\in [N_j,2N_j[}(f_j), \quad j=1,2,3,4 
\ena
one has the following bilinear estimates
\bnan
\label{inegbilin}
\left\|u_1 u_2\right\|_{L^2([0,T]\times M)} \leq C \min (N,L)^{s_0} \left\|f_1\right\|_{L^2(M)}\left\|f_2\right\|_{L^2(M)}\\
u_j(t)=e^{it\Delta}f_j,\quad j=1,2\nonumber
\enan
\end{hypo}
It is known to be true in the following examples ($1/2+$ means any $s>1/2$): \\
- $\mathbb{T}^3$ with $s_0=1/2+$, see \cite{Bourgain} \\
- the irrational torus $\R^3/(\theta_1 \Z \times \theta_2 \Z \times \theta_3\Z)$ with $\theta_i \in \R$, for which an estimate with $s_0=2/3+$ has been recently obtained in \cite{Bourgainirrat}. An easier proof for $s_0=3/4+$ can also be found in the beginning of \cite{Bourgainirrat} and in \cite{CatoireWang2008}\\
- $S^3$ with $s_0=1/2+$, see \cite{Xsbsphere}\\
- $S^2\times S^1$ with $s_0=3/4+$, see \cite{Xsbsphere}.

It yields some trilinear estimates in Bourgain spaces (see the definition below). 
For the control near a trajectory, we still have some particular assumptions that are again fulfilled in the particular geometries described above. Our result reads as follow 
\begin{theorem}
\label{thmcontrolenonlin} 
Let $T>0$ and $M$, $\omega$ such that Assumptions \ref{geometriccontrol}, \ref{hypinegXsb}, \ref{uniquecontinuationH1} and \ref{hypcommut} are fulfilled (see below). Let $1\geq s>s_0$ and $w \in \Xubt$ be a solution of
\begin{eqnarray}
\label{eqnnonlinpm}
\left\lbrace
\begin{array}{rcl}
i\partial_t w + \Delta w \pm|w|^2w&=&g\\
w(x,0)&=&w_{0}(x)
\end{array}
\right.
\end{eqnarray}
with $g\in C([0,T],H^1)$ supported in $[0,T]\times \overline{\omega}$.\\
Then, there exists $\varepsilon>0$, such that for every $u_0\in H^{s}$ with $\left\|u_0-w(0)\right\|_{H^s}<\varepsilon$, there exists $g_1\in C([0,T],H^s)$ supported in $[0,T]\times \overline{\omega} $ such that the unique solution $u$ in $\Xsbt$ of (\ref{eqnnonlinpm}) with $u(0)=u_0$ and $g$ replaced by $g_1$
fulfills $u(T)=w(T)$.\\
Moreover, for any $u_0\in H^1$ with $\left\|u_0-w(0)\right\|_{H^s}<\varepsilon$, the same conclusion holds with $g\in C([0,T],H^1)$.
\end{theorem}
An interesting fact is that the smallness assumption only concerns the $H^s$ norm, even if we want a control in $H^1$. For example, as in \cite{HumDehLeb}, if we assume $\nor{u_0}{H^1}\leq R_0$, we can find $N\in \N$ large enough such that the smallness asumption only concerns the $N$ first frequencies (see Corollary \ref{corpetitharmon}).

Let us describe the new hypothesis. Assumption \ref{uniquecontinuationH1} is a unique continuation result at weaker regularity.
\begin{hypo}
\label{uniquecontinuationH1}
Unique continuation in $H^1$: For every $T>0$, the only solution in $C([0,T],H^1)$ to the system 
\begin{eqnarray}
\left\lbrace
\begin{array}{c}
i\partial_t u + \Delta u + b_1(t,x)u+b_2(t,x)\overline{u}=0  \textnormal{ on } [0,T]\times M\\
u=0 \textnormal{ on } [0,T]\times \omega
\end{array}
\right.
\end{eqnarray}
where $b_1(t,x)$ and $b_2(t,x) \in L^{\infty}([0,T],L^{3})$ is the trivial one $u \equiv 0$.
\end{hypo}
We do not know if it is really stronger than Assumption \ref{uniquecontinuation} but for the moment, there are some example where we are able to prove Assumption \ref{uniquecontinuation} and not Assumption \ref{uniquecontinuationH1} using some weak Carleman estimates (see Appendix, Section \ref{sectuniqueness}). For instance, on $\Tot$, we are able to prove Assumption \ref{uniquecontinuation} for $\omega=\left\{x\in \R^3/\Z^3\left| x_1\in ]0,\varepsilon[+\Z \right.\right\}$ but not Assumption \ref{uniquecontinuationH1}. Yet, for the moment, we do not manage to deduce a controllability result from this statement.

The other new assumption is technical and yields quadrilinear estimates for a commutator
\begin{hypo}
\label{hypcommut}
There exists $C>0$ and $0\leq s_0<1$ such that for any $f_1,f_2,f_3,f_4\in L^2(M)$ satisfying
\bna
f_j=\mathbf{1}_{\sqrt{1-\Delta}\in [N_j,2N_j[}(f_j), \quad j=1,2,3,4 
\ena
one has the following quadrilinear estimate
\bnan
\label{quadrinlinschrod}
&&\sup_{\tau\in\R}\left|\int_{\R}\int_M \chi(t)e^{it\tau}u_1u_2\left((-\Delta)^{\varepsilon/2}u_3u_4-u_3(-\Delta)^{\varepsilon/2}u_4\right)dxdt\right| \\
&&\leq C(N_1^{\varepsilon}+N_2^{\varepsilon})\left(m(N_1,\troisp,N_4)\right)^{s_0} \nor{f_1}{L^2(M)}\nor{f_2}{L^2(M)}\nor{f_3}{L^2(M)}\nor{f_4}{L^2(M)}\nonumber\\
&&u_j(t)=e^{it\Delta}f_j,\quad j=1,2,3,4 \nonumber
\enan
where $\chi \in C^{\infty}_0(\R)$ is arbitrary and $m(N_1,\troisp,N_4)$ is the product of the smallest two numbers among $N_1, N_2, N_3, N_4$.\\ 
Moreover, the same result holds with $u_i$ replaced by $\overline{u_i}$ for $i$ in a subset of $\{1,2,3,4\}$.
\end{hypo}
For the three treated examples, we prove in Appendix, Section \ref{sectcommut}, that Assumption \ref{hypcommut} holds true with the same $s_0$ as in Assumption \ref{hypinegXsb}. We believe that it is the case for any manifold, but we did not manage to prove it.

As explained before, there are some examples for which we know that geometric control assumption is not necessary. For instance, for any pair of manifolds $M_1$, $M_2$ and $\omega_1\subset M_1$ such that $\omega_1$ satisfies observability estimate, $\omega_1 \times M_2$ satisfies observability estimate for the linear Schrödinger equation. We can then use this remark and the work of S. Jaffard \cite{Jaffard} and V. Komornik \cite{Komornik} for the linear equation on $\mathbb{T}^n$ to get some local nonlinear results . Since Theorem \ref{thmcontrolenonlin} is proved by a perturbative argument, we can also deduce controllability near $0$ from these already known linear control results. 
\begin{theorem}
\label{thmcontrolenonlinw0}
If $w\equiv0$ and $(M,\omega)$ is either :\\
-($\Tot$,any open set)\\
-($S^2\times S^1$, $\omega_1 \times S^1$) where $\omega_1$ is a neighborhood of the equator of $S^2$\\
-($S^2\times S^1$,$S^2\times ]0,\varepsilon[)$\\
Then, the same conclusion as Theorem \ref{thmcontrolenonlin} is true.
\end{theorem} 
\medskip

 Rosier and Zhang \cite{RosierZhangRectangle} have communicated to us that they simultaneoulsy obtained the same result for $\Tot$.
  
The proof of stabilization and of linear control with potential follows the same scheme as \cite{control-nl}. In a contradiction argument, we are led to prove the strong convergence to zero in $X^{s,b}_T$ of some weakly convergent sequence $(u_n)$ solution to damped NLS or Schrödinger with potential. Since the equation is subcritical, we use some linearisability properties of NLS in $H^1$ (see the work of P. Gérard \cite{linearisationondePG} for the wave equation). 

We first establish the strong convergence by some propagation of compactness. We adapt the argument of  \cite{control-nl} inspired by C. Bardos and T. Masrous \cite{BardosMasrour}. We use microlocal defect measures introduced by P. Gérard \cite{defectmeasure}. For a sequence $(u_n)$ weakly convergent to $0$ in $X^{s,b}_T$ satisfying
\begin{eqnarray*}
\left\lbrace
\begin{array}{c}
i\partial_t u_n + \Delta u_n \rightarrow 0 \quad \textnormal{in}\quad X^{s-1+b,-b}_T\\
a(x)u_n\rightarrow 0 \quad \textnormal{in}\quad L^2([0,T],H^s),
\end{array}
\right.
\end{eqnarray*}
we prove that $u_n\rightarrow 0$ in $L^2_{loc}([0,T],H^s)$. 

Once we know that the convergence is strong, we infer that the limit $u$ is solution to NLS. We would like to use Assumption \ref{uniquecontinuation} or \ref{uniquecontinuationH1} of unique continuation to prove that it is $0$. Yet, more regularity is needed to apply them. Again, we adapt the proof for $\Xsb$ spaces of propagation results of microlocal regularity coming from \cite{control-nl}.

\bigskip

In this article, $b'$ will be a constant such that estimates of Lemma \ref{lmtrilinXsb} holds. Actually, each of the trilinear estimates (with different $s$) that will be done will yield one $b'<1/2$ but remains true if we choose a greater one. So we take $b'<1/2$ as the largest of these constants. This allows to choose one $b>1/2$ with  $1>b+b'$.

In all the rest of the paper, $C$ will denote any constant whose value could change along the article.

\section{Some properties of $\Xsb$ spaces}

Since $M$ is compact, $\Delta$ has a compact resolvent and thus, the spectrum of $\Delta$ is discrete. We choose $e_k\in L^2(M)$, $k\in M$ an orthonormal basis of eigenfunctions of $-\Delta$, associated to eigenvalues $\lambda_k$. Denote $P_k$ the orthogonal projector on $e_k$. We equip the Sobolev space $H^s(M)$ with the norm (with $\left\langle x\right\rangle=\sqrt{1+|x|^2}$),
$$\left\|u\right\|^2_{H^s(M)}=\sum_{k}\left\langle \lambda_k\right\rangle^s \left\|P_k u\right\|^2_{L^2(M)}.$$
The Bourgain space $\Xsb$ is equipped with the norm
$$\left\|u\right\|^2_{X^{s,b}}= \sum_{k}\left\langle \lambda_k\right\rangle^s \left\|\left\langle \tau+ \lambda_k \right\rangle^b \widehat{P_k}(\tau) u\right\|^2_{L^2(\R_{\tau}\times M)}=\left\|u^{\#}\right\|^2_{H^{b}(\R,H^s(M))}$$
 where $u=u(t,x)$, $t\in\R$, $x\in M$, $u^{\#}(t)=e^{-it\Delta}u(t)$ and $ \widehat{P_k u}(\tau)$ denotes the Fourier transform of $P_k u$ with respect to the time variable.
 
\bigskip 

$\Xsbt$ is the associated restriction space with the norm
\bna 
\left\|u\right\|_{\Xsbt}=\inf \left\{ \left\| \tilde{u}\right\|_{\Xsb} \left| \tilde{u}=u \textnormal{   on   } ]0,T[\times M  \right. \right\}
\ena
We also write $\left\|u\right\|_{X^{s,b}_I}$ if the infinimum is taken on functions $\tilde{u}$ equalling $u$ on an interval $I$. The following properties  of $X^{s,b}_T$ spaces are easily verified.
\begin{enumerate}
 \item $X^{s,b}$ and $X^{s,b}_T$ are Hilbert spaces.
	\item If $s_1\leq s_2$, $b_1\leq b_2$ we have $X^{s_2,b_2} \subset  X^{s_1,b_1}$ with continuous embedding.
	\item For every $s_1<s_2$, $b_1<b_2$ and $T>0$, we have $X^{s_2,b_2}_T \subset  X^{s_1,b_1}_T$
with compact imbedding.
\item For $0<\theta <1$, the complex interpolation space $\left(X^{s_1,b_1},X^{s_2,b_2}\right)_{[\theta]}$ is $X^{(1-\theta)s_1+\theta s_2,(1-\theta)b_1+\theta b_2}$. \label{enuminterp}
\end{enumerate}
\ref{enuminterp}. can be proved with the interpolation theorem of Stein-Weiss for weighted $L^p$ spaces (see \cite{Bergh} p 114).

Then, we list some additional trilinear estimates that will be used all along the paper.
\begin{lemme}
\label{lmtrilinXsb}
If Assumption \ref{hypinegXsb} is fulfilled, for every $r\geq s>s_0$, there exist $0<b'<1/2$ and $C>0$ such that for any $u$ and $\tilde{u} \in X^{r,b'}$
\bnan
\label{multilinH1Hs}
 \left\| |u|^2 u \right\|_{X^{r,-b'}} &\leq& C \left\|u\right\|^2_{X^{s,b'}} \left\|u\right\|_{X^{r,b'}}\\
 \label{multilinH12Hs}
 \left\| |u|^2 \widetilde{u} \right\|_{X^{r,-b'}} &\leq& C \left\|u\right\|_{X^{s,b'}}\left\|u\right\|_{X^{r,b'}} \left\|\widetilde{u}\right\|_{X^{r,b'}}\\
\label{multilinH1}
 \left\||u|^2u-|\widetilde{u}|^2\widetilde{u}\right\|_{X^{s,-b'}} &\leq &C \left(\left\|u\right\|^2_{X^{s,b'}}+ \left\|\tilde{u}\right\|^2_{X^{s,b'}} \right) \left\|u-\tilde{u}\right\|_{X^{s,b'}}.
\enan
Moreover, the same estimates hold with $z_1\overline{z_2}z_3$ replaced by any $\R$-trilinear form on $\C$.
\end{lemme}
The proof of the previous lemma can be found in \cite{InventionesBGT}, \cite{gerardcourspise} or \cite{PGPierfelice}. Yet, in the Appendix, we prove some slightly different estimates, but the proof gives an idea of how Lemma \ref{lmtrilinXsb} is established. We also give some variants that will be used in the linearized version of our equations.
\begin{lemme}
\label{lmtrilinXsbbis}
If Assumption \ref{hypinegXsb} is fulfilled, for every $-1\leq s\leq 1$ and any $s_0<r\leq 1$, there exist $0<b'<1/2$ and $C>0$ such that for any $u\in X^{s,b'}$ and $a_1,a_2 \in X^{1,b'}$
\bnan
\label{inegtrilinmult}
\left\|a_1\overline{a_2}u\right\|_{X^{s,-b'}}\leq C\left\|a_1\right\|_{X^{1,b'}} \left\|a_2\right\|_{X^{1,b'}}\left\|u\right\|_{X^{s,b'}}\\
\label{inegtrilinmultcompact}
\left\||a_1|^2 u\right\|_{X^{s,-b'}}\leq C\left\|a_1\right\|_{X^{1,b'}} \left\|a_1\right\|_{X^{r,b'}}\left\|u\right\|_{X^{s,b'}}
\enan
Moreover, the same estimates hold with $z_1\overline{z_2}z_3$ replaced by any $\R$-trilinear form on $\C$.
\end{lemme}
\begin{proof}
We first prove (\ref{inegtrilinmultcompact}). Estimate (\ref{multilinH12Hs}) of Lemma \ref{lmtrilinXsb} implies that the operator of multiplication by $|a_1|^2$ maps $X^{1,b'}$ into $X^{1,-b'}$ with norm $\left\|a_1\right\|_{X^{1,b'}} \left\|a_1\right\|_{X^{r,b'}}$. IBy duality, it maps $X^{-1,b'}$ into $X^{-1,-b'}$ with the same norm. We get the same result for $-1\leq s\leq 1$ by interpolation, which yields (\ref{inegtrilinmultcompact}). For (\ref{inegtrilinmult}), we observe that estimate 
$$\left\|a_1\overline{a_2}u\right\|_{X^{1,-b'}}\leq C\left\|a_1\right\|_{X^{1,b'}} \left\|a_2\right\|_{X^{1,b'}}\left\|u\right\|_{X^{1,b'}}$$
holds whatever the position of the conjugate operator and we conclude similarly.
\end{proof}
Let us study the stability of the $\Xsb$ spaces with respect to some particular operations.
\begin{lemme}
\label{lemmetps}
Let $\varphi \in C^{\infty}_0(\R)$ and $u\in \Xsb$ then $\varphi(t) u \in \Xsb$.\\
If $u\in \Xsbt$ then we have $\varphi(t) u \in \Xsbt$.
\end{lemme}
\begin{proof}
We write
$$\left\|\varphi u\right\|_{X^{s,b}}= \left\|e^{-it\Delta}\varphi(t)u(t)\right\|_{H^{b}(\R,H^s)}= \left\|\varphi u^{\#}\right\|_{H^{b}(\R,H^s)}\leq C  \left\|u^{\#}\right\|_{H^{b}(\R,H^s)}\leq C\left\|u\right\|_{X^{s,b}}$$
We get the second result by applying the first one on any extension of $u$ and taking the infinimum.
\end{proof}
In the case of pseudodifferential operators in the space variable, we have to deal with a loss in $\Xsb$ regularity compared to what we could expect. Some regularity in the index $b$ is lost, due to the fact that a pseudodifferential operator does not keep the structure in time of the harmonics.\\
This loss is unavoidable as we can see, for simplicity on the torus $\mathbb{T}^1$ : we take  $u_n=\psi(t)e^{inx}e^{i|n^2|t}$ (where $\psi \in C^{\infty}_0$ equal to $1$ on $[-1,1]$) which is uniformly bounded in $X^{0,b}$ for every $b\geq 0$. Yet, if we consider the operator $B$ of order $0$ of multiplication by $e^{ix}$, we get $\left\|e^{ix}u_n\right\|_{X^{0,b}} \approx n^b$. Yet, we do not have such loss for operator of the form $(-\Delta)^r$ which acts from any $X^{s,b}$ to $X^{s-2r,b}$. But if we do not make any further assumption on the pseudodifferential operator, we can show that our example is the worst one : 
\begin{lemme}
\label{lemmepseudoxsb}
Let $ -1 \leq b \leq 1$ and $B$ be a pseudodifferential operator in the space variable of order $\rho$. For any $u\in \Xsb$ we have $B u \in X^{s-\rho-|b|,b}$.\\
Similarly, $B$ maps $\Xsbt$ into $X^{s-\rho-|b|,b}_T$.
\end{lemme}
\begin{proof}
We first deal with the two cases $b=0$ and $b=1$ and we will conclude by interpolation and duality.\\
For $b=0$, $X^{s,0}=L^2(\R,H^s)$ and the result is obvious.\\
For $b=1$, we have $u\in X^{s,1}$ if and only if
 $$u \in L^2(\R,H^s) \textnormal{ and } i\partial_t u+\Delta u \in L^2(\R,H^s)$$
 with the norm
 $$ \left\| u\right\|^2_{X^{s,1}}= \left\|u\right\|^2_{L^2(\R,H^s)}+\left\|i\partial_t u+\Delta u\right\|^2_{L^2(\R,H^s)}$$
Then, we have 
\bna
\left\|B u\right\|^2_{X^{s-\rho-1,1}}&=&\left\|Bu\right\|^2_{L^2(\R,H^{s-\rho-1})}+ \left\|i\partial_t Bu+\Delta Bu\right\|^2_{L^2(\R,H^{s-\rho-1})}\\
&\leq& C\left(\left\|u\right\|^2_{L^2(\R,H^{s-1})}+ \left\|B\left(i\partial_t u+\Delta u\right)\right\|^2_{L^2(\R,H^{s-\rho-1})}+ \left\|\left[B,\Delta \right]u\right\|^2_{L^2(\R,H^{s-\rho-1})}\right)\\
&\leq & C\left(\left\|u\right\|^2_{L^2(\R,H^{s-1})}+ \left\|i\partial_t u+\Delta u\right\|^2_{L^2(\R,H^{s-1})}+ \left\|u\right\|^2_{L^2(\R,H^{s})}\right)\\
&\leq & C\left\|u\right\|^2_{X^{s,1}}
\ena
Hence, $B$ maps $X^{s,0}$ into $X^{s-\rho,0}$ and $X^{s,1}$ into $X^{s-\rho-1,1}$. Then, we conclude by interpolation that $B$ maps $\Xsb=\left(X^{s,0},X^{s,1}\right)_{[b]}$ into {$\left(X^{s-\rho,0},X^{s-\rho-1,1}\right)_{[b]} =X^{s-\rho-b,b}$ } which yields the $b$ loss of regularity as announced.

By duality, this also implies that for $0\leq b \leq 1$, $B^*$ maps $X^{-s+\rho+b,-b}$ into $X^{-s,-b}$. As there is no assumption on $s\in \R$, we also have the result for $-1\leq b \leq 0$ with a loss $-b=|b|$.\\
To get the same result for the restriction spaces $\Xsbt$, we write the inequality for an extension $\tilde{u}$ of $u$, which yields
\bna
\left\|Bu\right\|_{X^{s-\rho-|b|,b}_T} \leq \left\|B\tilde{u}\right\|_{X^{s-\rho-|b|,b}} \leq C  \left\|\tilde{u}\right\|_{X^{s,b}} 
\ena
Taking the infinimum on all the $\tilde{u}$, we get the claimed result.
\end{proof}
We will also use the following elementary estimate (see e.g. \cite{ginibre} or \cite{Bourgain}).
\begin{lemme}
\label{gainint}
Let $(b,b')$ satisfying
\begin{eqnarray}
0<b'<\frac{1}{2}<b,~~~~b+b'\leq 1. 
\end{eqnarray}
If we note $F(t)=\Psi\left(\frac{t}{T}\right)\intt f(t')dt'$, we have for $T\leq 1$
\bna
\left\|F\right\|_{H^b} \leq CT^{1-b-b'}\left\|f\right\|_{H^{-b'}}.
\ena
\end{lemme}
In the futur aim of using a boot-strap argument, we will need some continuity in $T$ of the $X^{s,b}_T$ norm of a fixed function : 
\begin{lemme}
\label{continuiteXsbt}
Let $0<b<1$ and $u$ in $X^{s,b}$ then the function
\begin{eqnarray*}
\left\lbrace
\begin{array}{rcrcl}
f&:&]0,T]& \longrightarrow &  \R \\
 & &t    & \longmapsto     & \left\|u\right\|_{X^{s,b}_{t}}
\end{array}
\right.
\end{eqnarray*} 
 is continuous.
Moreover, if $b>1/2$, there exists $C_b$ such that 
$$\lim_{t\rightarrow 0} f(t) \leq C_b \nor{u(0)}{H^s}.$$
\end{lemme}
\begin{proof}
By reasoning on each component on the basis, we are led to prove the result in $H^b(\R)$. The most difficult case is the limit near $0$. It suffices to prove that if $u\in H^b(\R)$, with $b>1/2$, satisfies $u(0)=0$, and $\Psi \in C^{\infty}_0(\R)$ with $\Psi(0)=1$, then
$$\Psi\left(\frac{t}{T}\right)u \tend{T}{0} 0 \quad \textnormal{in} \quad H^b.$$
Such a function $u$ can be written $\intt f$ with $f\in H^{b-1}$. Then, Lemma \ref{gainint} gives the result we want if $u\in H^{b+\varepsilon}$. Nevertheless, if we only have $u\in H^b$, $\Psi(\frac{t}{T})u$ is uniformly bounded. We conclude by a density argument.\end{proof}
The following lemma will be useful to control solutions on large intervals that will be obtained by piecing together solutions on smaller ones. We state it without proof.
\begin{lemme}
\label{lemmerecouvrement}
Let $0<b<1$. If $\bigcup ]a_k,b_{k}[$ is a finite covering of $[0,1]$, then there exists a constant $C$ depending only of the covering such that for every $u\in X^{s,b}$
\begin{eqnarray*}
\left\|u\right\|_{X^{s,b}_{[0,1]}}\leq C\sum_k \left\|u\right\|_{X^{s,b}_{[a_k,b_{k}]}}.
\end{eqnarray*}
\end{lemme}
\section{Existence of solution to NLS with source and damping term}
\subsection{Nonlinear equation}
\label{sectexistence}
Let $a \in C^{\infty}(M)$ taking real values fixed.\\
We will prove the existence for defocusing non linearity of degree $3$ : they will have the form $\alpha u+ \beta |u|^2u$, with $\alpha, \beta \geq 0$. 
\begin{prop}
\label{thmexistenceNl}
Let $T>0$ and $s\geq 1$.
Assume that $M$ satisfies Hypothesis \ref{hypinegXsb}.
Then, for every $g\in L^2([0,T],H^s)$ and $u_0 \in H^s$, there exists a unique solution $u$ on $[0,T]$ in $X^{s,b}_T$ to the Cauchy problem 
\begin{eqnarray}
\label{dampedeqn}
\left\lbrace
\begin{array}{rcl}
i\partial_t u + \Delta u -\alpha u-\beta|u|^2u&=&a(x)(1-\Delta)^{-1}a(x)\partial_t u +g \textnormal{ on } [0,T]\times M\\
u(0)&=&u_{0} \in H^s
\end{array}
\right.
\end{eqnarray}
Moreover the flow map $$
\begin{array}{rcrcl}
F &:& H^s(M) \times L^2([0,T],H^s(M))&\rightarrow & X^{s,b}_{T}\\
           & & (u_0,g) &\mapsto   &  u
\end{array}
$$ is Lipschitz on every bounded subset.
\end{prop}
\begin{proof}
It is strongly inspired by the one of Bourgain \cite{Bourgain} and Dehman, Gérard, Lebeau \cite{control-nl} for the stabilization term. The proof is mainly based on estimates of Lemma \ref{lmtrilinXsb}.

First, we establish that the operator $J$ defined by $Jv=(1+ia(x)(1-\Delta)^{-1}a(x))v$ is an isomorphism of $H^{s}$ and $\Xsb$ ($s \in \R$ and $-1\leq b \leq 1$ ). 

$J$ is an isomorphism of $L^2$ because of its decomposition in identity plus an antiselfadjoint part $J=1+A$. It is then an isomorphism of $H^s$ with $s\geq 0$ by ellipticity and for every $s\in \R$ by duality. Using Lemma \ref{lemmepseudoxsb}, we infer that if $-1\leq b \leq 1$, $A$ maps $\Xsb$ into itself. Moreover, $J^{-1}$ (considered for example acting on $L^2([0,T]\times M)$) is a pseudodifferential operator of order $0$ and satisfies $J^{-1}=1-AJ^{-1}$. Then, using again Lemma \ref{lemmepseudoxsb}, we get that $AJ^{-1}$ maps $\Xsb$ into $X^{s-|b|+2}$ and $J$ is an isomorphism of $\Xsb$.

In the sequel of the proof, $v$ will denote $Ju$. Hence, we can write system (\ref{dampedeqn}) as 
\begin{eqnarray}
\label{eqnR0damped}
\left\lbrace
\begin{array}{rcl}
\partial_t v -i \Delta v -R_0 v+i\beta|u|^2u&=&-ig\textnormal{ on } [0,T]\times M\\
v&=&Ju\\
v(0)&=&v_{0}=Ju_0 \in H^s
\end{array}
\right.
\end{eqnarray}
where $R_0=-i\Delta AJ^{-1}+i\alpha J^{-1}$ is a pseudo-differential operator of order 0.

First, we notice that if $g\in L^2([0,T],H^s)$, it also belongs to $X^{s,-b'}_T$ as $b'\geq 0$.\\
We consider the functional
 $$\Phi(v)(t)=e^{it\Delta}v_0+\int_0^t e^{i(t-\tau)\Delta}\left[R_0 v-i\beta \left|u\right|^2u-i g\right](\tau) d\tau $$
We will apply a fixed point argument on the Banach space $\Xsbt$.
Let $\psi \in C^{\infty}_0(\R)$ be equal to $1$ on $[-1,1]$. Then by construction, (see \cite{ginibre}) :
$$\left\|\psi(t)e^{it\Delta}v_0\right\|_{\Xsb}=\left\|\psi\right\|_{H^{b}(\R)} \left\|v_0\right\|_{H^s}$$
Thus, for $T\leq 1$ we have
$$\left\|e^{it\Delta}v_0\right\|_{\Xsbt}\leq C \left\|v_0\right\|_{H^s}\leq C \left\|u_0\right\|_{H^s}$$
For $T\leq 1$, the one dimensional estimate of Lemma \ref{gainint} implies
$$ \left\|\psi(t/T)\intt e^{i(t-\tau)\Delta}F(\tau)\right\|_{\Xsb}\leq CT^{1-b-b'}\left\|F\right\|_{X^{s-b'}}$$
and then
\bnan
\label{inegprincip}
\left\|\int_0^t e^{i(t-\tau)\Delta}\left[R_0 v-i\beta \left|u\right|^2u-i g\right](\tau) d\tau\right\|_{\Xsbt}& \leq & CT^{1-b-b'}\left\|R_0 v-\beta i\left|u\right|^2u-i g\right\|_{X^{s,-b'}_T}\nonumber\\
&\leq & CT^{1-b-b'}\left\|R_0 v\right\|_{X^{s,0}_T}+\left\|\left|u\right|^2u\right\|_{X^{s,-b'}_T}+\left\|g\right\|_{X^{s,-b'}_T}\nonumber\\
& \leq & CT^{1-b-b'}\left\|v\right\|_{X^{s,b}_T}\left(1+\left\|v\right\|_{X^{1,b}_T}^2\right)+\left\|g\right\|_{X^{s,-b'}_T}
\enan
Thus
\bnan
\label{tameestimateDuhamel}
\left\|\Phi(v)\right\|_{\Xsbt}\leq C \left\|u_0\right\|_{H^s}+\left\|g\right\|_{X^{s,-b'}_T}+CT^{1-b-b'}\left\|v\right\|_{X^{s,b}_T}\left(1+\left\|v\right\|_{X^{1,b}_T}^2\right)
\enan
and similarly,
\bnan
\label{majdiff}
\left\|\Phi(v)-\Phi(\tilde{v})\right\|_{\Xsbt}\leq CT^{1-b-b'}\left\|v-\tilde{v}\right\|_{\Xsbt}\left(1+\left\|v\right\|_{\Xsbt}^2+\left\|\tilde{v}\right\|_{\Xsbt}^2\right)
\enan
These estimates imply that if $T$ is chosen small enough $\Phi$ is a contraction on a suitable ball of $\Xsbt$. Moreover, we have uniqueness in the class $\Xsbt$ for the Duhamel equation and therefore for the Schrödinger equation. 

We also prove propagation of regularity.\\
If $u_0\in H^s$, with $s>1$, we have an existence time $T$ for the solution in $\Xubt$ and another time $\tilde{T}$ for the existence in $X^{s,b}_{\tilde{T}}$. By uniqueness in $\Xubt$, the two solutions are the same on $[0,\tilde{T}]$. Assume $\tilde{T}<T$. Then, $\left\|u(t,.)\right\|_{H^s}$ explodes as $t$ tends to $\tilde{T}$ whereas $\left\|u(t,.)\right\|_{H^1}$ remains bounded. Using local existence in $H^1$ and Lemma \ref{lemmerecouvrement}, we get that $\left\|u\right\|_{X^{1,b}_{\tilde{T}}}$ is finite. Applying tame estimate (\ref{tameestimateDuhamel}) on a subinterval $[T-\varepsilon,T]$, with $\varepsilon$ small enough such that $C\varepsilon^{1-b-b'}\left(1+\left\|v\right\|_{X^{1,b}_T}^2\right)<1/2$, we obtain 
\bna
\left\|v\right\|_{\Xsbt}\leq C \left\|u(T-\varepsilon)\right\|_{H^s}+\left\|g\right\|_{X^{s,-b'}_T}.
\ena
Therefore, $u\in X^{s,b}_{\tilde{T}}$, and this contradicts the explosion of $\left\|u(t,.)\right\|_{H^s}$ near $\tilde{T}$.

Next, we use energy estimates to get global existence.\\
First, we will consider the energy :
\bna
E(t)=\frac{1}{2}\int_M \left|\nabla u \right|^2 +\frac{1}{2}\alpha \int_M \left|u\right|^2+\beta	\frac{1}{4}\int_M \left|u\right|^4
\ena
The energy is conserved if $g=0$ and $a=0$. It is nonincreasing if $g=0$. In general, multiplying our equation by $\partial_t \bar{u}$, we have the relation : 
\begin{eqnarray*}
E(t)-E(0)&=&-\int_{0}^{t}\left\|(1-\Delta)^{-1/2}a(x)\partial_t u\right\|^2_{L^2}-\Re \int_{0}^{t}\int_M  g \overline{\partial_t u} \\
&=&-\int_{0}^{t}\left\|(1-\Delta)^{-1/2}a(x)\partial_t u\right\|^2_{L^2}-\Re \int_{0}^{t}\int_M  (J^{-1*}g) \overline{\partial_t v}\\
&=&-\int_{0}^{t}\left\|(1-\Delta)^{-1/2}a(x)\partial_t u\right\|^2_{L^2}-\Re \int_{0}^{t}\int_M  (J^{-1*}g) \overline{i \Delta v +R_0 v-i\beta|u|^2u-ig}\\
\end{eqnarray*}

If $0\leq t \leq T$ (for this equation, there is not global existence in negative time) and $\beta > 0$, we get 
\begin{eqnarray*}
E(t) & \leq &E(0))+ C \int_{0}^{t} \left\|\nabla (J^{-1*}g) \right\|_{L^2} \left\| \nabla u\right\|_{L^2}+ \int_{0}^{t} \left\|g\right\|_{L^2}\left\|u\right\|_{L^2}\\
& & +\int_{0}^{t} \left\|g\right\|_{L^4}\left\|u\right\|^3_{L^4}+\nor{g}{L^2([0,T]\times M)}^2\\
&\leq & E(0)+ C\int_{0}^{t} \left\|g(\tau) \right\|_{H^1}  \sqrt{E(\tau)} +C\int_{0}^{t} \left\|g(\tau)\right\|_{L^2} \left(E(\tau)\right)^{1/4}\\
& &+C\int_{0}^{t} \left\|g(\tau)\right\|_{H^1} \left(E(\tau)\right)^{3/4}+\nor{g}{L^2([0,T]\times M)}^2\\
&\leq & E(0)+ C \int_{0}^{t} \left\|g(\tau) \right\|_{H^1} \left[1+ \left(E(\tau)\right)^{3/4}\right]+\nor{g}{L^2([0,T]\times M)}^2
\end{eqnarray*}
Therefore
\begin{eqnarray*}
\underset{0\leq \tau \leq t }{\max} E(\tau) \leq E(0)+ C (\left[1+ \underset{0\leq \tau \leq t }{\max}E(\tau) ^{3/4}\right]  \left\|g\right\|_{L^1([0,T],H^1)}+\nor{g}{L^2([0,T]\times M)}^2
\end{eqnarray*}
So we have finally 
\begin{eqnarray}
\label{majorationenergie}
E(t) \leq C\left( 1+E(0)^4 +\nor{g}{L^2([0,T]\times M)}^8+ \left\|g\right\|^4_{L^1([0,T],H^1)}\right)
\end{eqnarray}
This implies that the energy is bounded if $g\in L^2([0,T],H^1)$ and yields global existence in $\Xubt$ for every $T>0$. The fact that the flow is locally Lipschitz follows from estimate (\ref{majdiff}).
\end{proof}
\begin{remarque}
If $g=0$, the solution of (\ref{dampedeqn}) satisfies the energy decay 
\begin{eqnarray*}
E(t)-E(0)=-\int_{0}^{t}\left\|(1-\Delta)^{-1/2}a(x)\partial_t u\right\|^2_{L^2}
\end{eqnarray*}
This is obtained for initial data in $H^2$ by multiplying the equation by $\partial_t \bar{u}$ and can be extended to initial data in $H^1$ by approximation.
\end{remarque}

\begin{remarque}
\label{rmkbound}
We have also proved that for any $u_0$, $g$ with $\nor{u_0}{H^1}+\nor{g}{L^2([0,T],H^1)}\leq A $, the solution $u$ of (\ref{dampedeqn}) satisfies
$$\nor{u}{\Xubt}\leq C(T,A).$$
\end{remarque}

\begin{remarque}
\label{gainb}
If we look carefully at inequality (\ref{inegprincip}), we see that we have for $0<\varepsilon< 1-b-b'$
 \bnan
\left\|\int_0^t e^{(t-\tau)\Delta}\left[R_0 v-i\left|u\right|^2u-i Jg\right](\tau) d\tau\right\|_{X^{1,b+\varepsilon}}& \leq & CT^{1-b-b'-\varepsilon}\left\|R_0 v-i\left|u\right|^2u-i Jg\right\|_{X^{1,-b'}_T}\nonumber\\
& \leq & CT^{1-b-b'-\varepsilon}\left\|v\right\|_{X^{1,b}_T}\left(1+\left\|v\right\|_{X^{1,b}_T}^2\right)+\left\|g\right\|_{L^2([0,T],H^1)}
\enan
And we can then conclude that $u$ is bounded in $X^{1,b+\varepsilon}_T$.
\end{remarque}
\begin{remarque}
We notice that for a solution of the equation, the term of stabilization $a(x)(1-\Delta)^{-1}a(x)\partial_t u$ belongs to $L^{\infty}([0,T],H^1(M))$ as expected. Actually, for a solution, this term acts as an operator of order $0$. This is more visible using the equation fulfilled by $v=Ju$.
\end{remarque}
Then, in the aim of obtaining controllability near trajectories, we prove an appropriate existence theorem.
\begin{prop}
\label{thmexistproche}
Suppose that Assumption \ref{hypinegXsb} is fulfilled. Let $T>0$ and $w$ solution in $\Xubt$ of 
\begin{eqnarray}
\label{eqnproche}
\left\lbrace
\begin{array}{rcl}
i\partial_t w + \Delta w &=&\pm|w|^2w+g_1 \textnormal{ on } [0,T]\times M\\
w(0)&=&w_{0} \in H^1
\end{array}
\right.
\end{eqnarray}
with $g_1\in L^{2}([0,T],H^1)$. Then, for any $s\in ]s_0,1]$, there exists $\varepsilon>0$ such that for any $u_0\in H^s$ and $g\in L^2([0,T],H^s)$ with $\left\|u_0-w_0\right\|_{H^s}+\left\|g_1-g\right\|_{L^2([0,T],H^s)}\leq  \varepsilon $ there exists a unique solution $u$ in $\Xsbt$ of equation (\ref{eqnproche}).\\
Moreover for any $1\geq r\geq s$ there exists $C=C(r,\left\|w\right\|_{\Xubt},T)>0$ such that, if $u_0\in H^{r}$ and $g\in L^2([0,T],H^{r})$, we have $u\in X^{r,b}_T$ and
\bnan
\label{ineglinearbehaviorproche}
\left\|u-w\right\|_{X^{r,b}_T} \leq  C\left( \left\|u_0-w_0\right\|_{H^{r}}+\left\|g_1-g\right\|_{L^2([0,T],H^{r})}\right).
\enan 
\end{prop}
\begin{remarque}
In the focusing case, the existence of $w$ is not guaranteed for any $w_0$ $g_1$ and $T$, and the result we prove assumes this existence.   
\end{remarque}
\begin{remarque}
Here, we emphasize the fact that the asumption of smallness only concerns the $H^s$ norm and not $H^r$. This is a consequence of the subcritical behavior.
\end{remarque}
\begin{proof}
We want to linearize the equation. If $u=w+r$ and $g=g_1+g_r$, then
\bna
\left|w+r\right|^2(w+r)&=&\left|w\right|^2w+2\left|w\right|^2r+w^2\bar{r}+2\left|r\right|^2w+r^2\bar{w} +\left|r\right|^2r\\
&=&\left|w\right|^2w+2\left|w\right|^2r+w^2\bar{r}+F(w,r).
\ena
We are looking for $r$ solution of 
\begin{eqnarray}
\left\lbrace
\begin{array}{rcl}
i\partial_t r + \Delta r &=&2\left|w\right|^2r+w^2\bar{r}+F(w,r)+g_r\\
r(x,0)&=&r_{0}(x)
\end{array}
\right.
\end{eqnarray}
We make a proof similar to Proposition \ref{thmexistenceNl}. We only write the necessary estimates. (\ref{multilinH1Hs}) and (\ref{multilinH12Hs}) yield
\bna
\left\|r\right\|_{X^{r,b}_T}&\leq& C\left(\left\|r_0\right\|_{H^{r}}+\left\|g_r\right\|_{L^2([0,T],H^{r})}\right)+CT^{1-b-b'}\left\|w\right\|_{\Xubt}^2\left\|r\right\|_{X^{r,b}_T}\\
&&+CT^{1-b-b'}\left(\left\|w\right\|_{\Xubt}\left\|r\right\|_{X^{r,b}_T}\left\|r\right\|_{X^{s,b}_T}+\left\|r\right\|_{X^{r,b}_T}\left\|r\right\|_{X^{s,b}_T}^2\right).
\ena
With $T$ such that $CT^{1-b-b'}\left\|w\right\|_{\Xubt}^2 <1/2$, it yields
\bnan
\left\|r\right\|_{X^{r,b}_T}&\leq& C\left(\left\|r_0\right\|_{H^{r}}+\left\|g_r\right\|_{L^2([0,T],H^{r})}\right)\nonumber \\
\label{estimlinproche}
&&+CT^{1-b-b'}\left(\left\|w\right\|_{\Xubt}\left\|r\right\|_{X^{r,b}_T}\left\|r\right\|_{X^{s,b}_T}+\left\|r\right\|_{X^{r,b}_T}\left\|r\right\|_{X^{s,b}_T}^2\right).
\enan
First, we apply this with $r=s$. As we have proved in Lemma \ref{continuiteXsbt} the continuity with respect to $T$ of $\left\|r\right\|_{X^{s,b}_T}$ we are in position to apply a boot-strap argument : for $\left\|r_0\right\|_{H^s}+\left\|g_r\right\|_{L^2([0,T],H^s)}$ small enough (depending only on $\left\|w\right\|_{\Xubt}$), we obtain  : 
\bna
\left\|r\right\|_{X^{s,b}_T}\leq C \left\|r_0\right\|_{H^s}+\left\|g_r\right\|_{L^2([0,T],H^s)}.
\ena
Repeating the argument on every small interval, using that $\left\|r\right\|_{X^{s,b}_T}$ controls $L^{\infty}(H^s)$ and matching solutions with Lemma \ref{lemmerecouvrement}, we get the same result for every large interval, with a smaller constant $\varepsilon$, depending only on $s$, $T$ and $\left\|w\right\|_{\Xubt} $.\\
Then, we return to the general case $r\geq s$ and $CT^{1-b-b'}\left\|w\right\|_{\Xubt}^2 <1/2$. For $T$ small enough (depending only on $r$, $\varepsilon$ and $\left\|w\right\|_{\Xubt} $), estimate (\ref{estimlinproche}) becomes 
\bna
\left\|r\right\|_{X^{r,b}_T}\leq C\left( \left\|r_0\right\|_{H^{r}}+\left\|g_r\right\|_{L^2([0,T],H^{r})}\right)
\ena
Again, we obtain the desired result by piecing solutions together.
\end{proof}
\subsection{Linear equation with rough potential}
The control near trajectories will be obtained by a perturbation of control of linear Schrödinger equation with rough potential. The equation considered are the linearization of nonlinear equations and its dual version. We establish here the necessary estimates.
\begin{prop}
\label{thmlineareqn}
Suppose Assumption \ref{hypinegXsb}. Let $T>0$, $s\in [-1,1]$, $A>0$ and $w\in \Xubt$ with $\nor{w}{\Xubt}\leq A$.\\
For every $u_0\in H^{s}$ and $g\in X^{s,-b'}_T$ there exists a unique solution $u$ in $X^{s,b}_T$ of equation 
\begin{eqnarray}
\left\lbrace
\begin{array}{rcl}
i\partial_t u + \Delta u &=&\pm2|w|^2u\pm w^2\overline{u}+g \textnormal{ on } [0,T]\times M\\
u(0)&=&u_{0} \in H^{s}
\end{array}
\right.
\end{eqnarray}
Moreover there exists $C=C(s,A,T)>0$ such that
\bnan
\label{inegeqnlin}
\left\|u\right\|_{X^{s,b}_T} \leq  C\left( \left\|u_0\right\|_{H^{s}}+\left\|g\right\|_{X^{s,-b'}}\right).
\enan 
\end{prop}
\begin{proof}
We make the same arguments as above using estimates of Lemma \ref{lmtrilinXsbbis}.
\end{proof}
\section{Linearisation in $H^1$}
The following result show that any sequence of solutions with Cauchy data weakly convergent to $0$ asymptotically behave as solutions of the linear equation. These types of results were first introduced by P. Gérard in  \cite{linearisationondePG} for the wave equation and are typical of subcritical situations.
\begin{prop}
\label{proplinearisation}
Suppose Assumption \ref{hypinegXsb} is fulfilled. Let $(u_n)\in \Xubt$ be a sequence of solutions of
\begin{eqnarray}
\label{eqncontrolenl}
\left\lbrace
\begin{array}{rcl}
i\partial_t u_n + \Delta u_n -u_n-|u_n|^2u_n&=&a(x)(1-\Delta)^{-1}a(x)\partial_t u_n \textnormal{ on } [0,T]\times M\\
u_n(0)&=&u_{n,0} \in \Hu
\end{array}
\right.
\end{eqnarray}
such that
$$u_{n,0} \underset{\Hu}{\rightharpoonup} 0.$$ 
Then 
$$|u_n|^2u_n \underset{X^{1,-b'}_T}{\longrightarrow} 0.$$
\end{prop}

\begin{proof}
We prove that any subsequence (still denoted $u_n$) admits another subsequence converging to $0$. The main point is the tame $\Xsbt$ estimate of Lemma \ref{lmtrilinXsb}. For one $s_0<s<1$ we have
\bnan
\label{ineglinea}
\left\| |u_n|^2u_n \right\|_{X^{1,-b'}_T}\leq \left\|u_n\right\|^2_{X^{s,b}_T}\left\|u_n\right\|_{\Xubt}
\enan
First, using Remark \ref{rmkbound}, we conclude that $u_n$ is bounded in $X^{1,b}_T$, and actually by Remark \ref{gainb}, $u_n$ is bounded in $X^{1,b+\varepsilon}_T$ for some $\varepsilon>0$. By compact embedding of $X^{1,b+\varepsilon}_T$ into $X^{s,b}_T$ we obtain that $u_n$ admits a subsequence converging weakly in $X^{1,b}_T$ and strongly in $X^{s,b}_T$ to a function $u\in X^{s,b}_T$ with $u(0)=0$. $u_n(0)$ strongly converges to $0$ in $H^s$ and by continuity of the nonlinear flow in $H^s$, $u_n$ strongly converges to $0$ in $X^{s,b}_T$.
This yields the desired result thanks to (\ref{ineglinea}).\end{proof}
\section{Propagation of compactness}
In this section, we adapt some theorems of Dehman-Gerard-Lebeau \cite{control-nl} in the case of $\Xsb$ spaces. We recall that $S^*M$ denotes the cosphere bundle of the Riemannian manifold $M$,
$$S^*M =\left\{(x,\xi) \in T^*M\quad :\quad \left|\xi \right|_x=1\right\}$$
\begin{prop}
\label{thmpropagation2}Let $r\in \R$.
Let $u_n$ be a sequence of solutions to 
$$i\partial_t u_n + \Delta u_n=f_n$$  
such that for one $0\leq b\leq 1$, we have
$$ \left\|u_n\right\|_{X^{r,b}_T} \leq C,~~\left\|u_n\right\|_{X^{r-1+b,-b}_T} \rightarrow 0~~ and~~\left\|f_n\right\|_{X^{r-1+b,-b}_T}\rightarrow  0 $$ 
Then, there exists a subsequence $(u_{n'})$ of $(u_n)$ and a positive measure $\mu$ on $]0,T[\times S^*M$ such that for every tangential pseudodifferential operator $A=A(t,x,D_x)$ of order $2r$ and of principal symbol $\sigma(A)=a_{2r}(t,x,\xi)$,
$$  (A(t,x,D_x)u_{n'},u_{n'})_{L^2(]0,T[\times M)} \rightarrow \int_{]0,T[\times S^*M} a_{2r}(t,x,\xi)~d\mu(t,x,\xi)$$
Moreover, if $G_s$ denotes the geodesic flow on $S^*M$, one has for every $s\in \R$,
$$G_s(\mu)=\mu$$
\end{prop}
\begin{proof}
Existence of the measure : it is based on Gärding inequality, see \cite{defectmeasure} for an introduction.\\
Propagation : 
Denote $L$ the operator $L  = i\partial_t  + \Delta$.  
Let $\varphi=\varphi(t)\in C_0^\infty(]0,T[)$, $B(x,D_x)$ be a pseudodifferential operator of order $1$, with principal symbol  $b_{2r-1}$, $A(t,x,D_x)=\varphi(t)B(x,D_x)$. For $\varepsilon >0$, we denote $A_{\varepsilon}=\varphi B_{\varepsilon}=Ae^{\varepsilon\Delta}$ for the regularization.

As $A_{\varepsilon}u_n$ and $A^*_{\varepsilon}u_n$ are $C^{\infty}$, we can write $(L u_n,A^*_{\varepsilon}u_n)_{L^2(]0,T[\times M)}=(f_n,A^*_{\varepsilon}u_n)_{L^2(]0,T[\times M)}$ and\\ $(A_{\varepsilon}u_n,Lu_n)_{L^2(]0,T[\times M)}=(A_{\varepsilon}u_n,f_n)_{L^2(]0,T[\times M)}$. We write by a classical way
\begin{eqnarray*}
\alpha_{n,\varepsilon}&=& (L u_n,A^*_{\varepsilon}u_n)_{L^2(]0,T[\times M)}-(A_{\varepsilon}u_n,Lu_n,)_{L^2(]0,T[\times M)}\\
&=& ([A_{\varepsilon},\Delta]u_n,u_n)-i(\partial_t(A_{\varepsilon}) u_n,u_n)
\end{eqnarray*}
We will strongly use Lemma \ref{lemmetps} and \ref{lemmepseudoxsb} without citing them.

$\partial_t(A_{\varepsilon})$ is of order $2r-1$ uniformly in $\varepsilon$, then, 
\bna
\sup_{\varepsilon}(\partial_t(A_{\varepsilon}) u_n,u_n)_{L^2(]0,T[\times M)} \leq C  \|\partial_t(A_{\varepsilon}) u_n\|_{X^{-r+1-b,b}_T}\|u_n\|_{X^{r-1+b,-b}_T} \leq C \|u_n\|_{X^{r,b}_T}\|u_n\|_{X^{r-1+b,-b}_T} 
\ena
which tends to $0$ if $n\rightarrow \infty$.\\
But we have also
\begin{eqnarray*}
\alpha_{n,\varepsilon}&=& (f_n,A^*_{\varepsilon}u_n)_{L^2(]0,T[\times M)}-(A_{\varepsilon}u_n,f_n)_{L^2(]0,T[\times M)}
\end{eqnarray*}

\begin{eqnarray*}
\left|(f_n,A^*_{\varepsilon}u_n)_{L^2(]0,T[\times M)}\right| &\leq& \|f_n\|_{X^{r-1+b,-b}_T} \|A^*_{\varepsilon}u_n\|_{X^{-r+1-b,b}_T}\\
&\leq& \|f_n\|_{X^{r-1+b,-b'}_T} \|u_n\|_{X^{r,b}_T}
\end{eqnarray*}

Then, $\sup_{\varepsilon}\left|(f_n,A^*_{\varepsilon}u_n)_{L^2(]0,T[\times M)}\right|\rightarrow 0$ when $n \rightarrow \infty$. The same estimate for the other terms gives $\sup_{\varepsilon} \alpha_{n,\varepsilon} \rightarrow 0$.

Finally, taking the supremum on $\varepsilon$ tending to $0$, we get
  $$(\varphi[B,\Delta]u_n,u_n)_{L^2(]0,T[\times M)} \rightarrow 0 \textnormal{ when } n \rightarrow \infty $$
 which means, in terms of measure
  $$\int_{]0,T[\times S^*M} \varphi(t) \left\{\sigma_{2}(\Delta),b_{2r-1}  \right\} ~d\mu(t,x,\xi)  =0. $$
 This is precisely the propagation along the geodesic flow.
\end{proof}
  
\begin{corollaire}
\label{corrpropag2}
Let $r\in\R$. Assume that $\omega \subset M$ satisfies Assumption $\ref{geometriccontrol}$ and $a \in C^{\infty}(M)$, as in  (\ref{lienoma}). Let $u_n$ be a sequence bounded in $X^{r,b'}_T$ with $0< b'<1/2$, weakly convergent to $0$ and satisfying
\begin{eqnarray}
\left\lbrace
\begin{array}{c}
i\partial_t u_n + \Delta u_n \rightarrow 0 \textnormal{ in } X^{r,-b'}_T\\
a(x)u_n \rightarrow 0 \textnormal{ in } L^2([0,T],H^{r})
\end{array}
\right.
\end{eqnarray}
Then, we have $u_n \rightarrow 0 $ in $X^{r,1-b'}$.
\end{corollaire}

\begin{proof} Let $(u_{n_k})$ be a quelconque subsequence of $(u_n)$. The asumption on $b'$ and compact embedding  allow us to apply Proposition \ref{thmpropagation2}. We can attach to $(u_{n_k})$ a microlocal defect measure in $L^2([0,T],H^r)$ that propagates along the geodesics with infinite speed. The second assumption gives $a(x)\mu=0$.
  By Assumption \ref{geometriccontrol}, and the fact that $a$ is elliptic on $\omega$, we have $\mu=0$ on $]0,T[\times S^*M$, ie $(u_{n'}) \rightarrow 0 $ in $L^{2}([0,T],H^r)$, and  $u_n \rightarrow u $ in $L^{2}([0,T],H^r)$.\\
Then, we can pick $t_0$ such that $u_n(t_0)\rightarrow 0$ in $H^r$.\\
Using Lemma \ref{gainint} and asumptions on $b'$, we get for $T\leq 1$
\bna
\left\| \int_0^t e^{i(t-\tau)\Delta}f_n(\tau) d\tau \right\|_{X^{r,1-b'}_T}\leq C \left\|f_n\right\|_{X^{r,-b'}_T}
\ena
Using Duhamel formula, we conclude $u_n\rightarrow 0$ in $X^{r,1-b'}_T$.\\
Then, the hypothesis $T\leq 1$ is easily removed by piecing solutions together as in Lemma \ref{lemmerecouvrement}.
\end{proof}
\section{Propagation of regularity}

We write Proposition 13 of \cite{control-nl} with some $\Xsb$ asumptions on the second term of the equation.
\begin{prop}
\label{propreg}
Let $T>0$, $0\leq b<1$ and $u\in X^{r,b}_T $, $r \in \R$ solution of
$$i \partial_t u +\Delta u=f \in X^{r,-b}_T $$
Given $\omega_0=(x_0,\xi_0) \in T^*M \setminus 0$, we assume that there exists a $0-order$ pseudo-differential operator $\chi(x,D_x)$, elliptic in $\omega_0$ such that
$$\chi(x,D_x)u \in L^2_{loc}(]0,T[,H^{r+\rho}) $$
for some $\rho \leq   \frac{1-b}{2}$. Then, for every $\omega_1 \in \Gamma_{\omega_0}$, the geodesic ray starting at $\omega_0$, there exists a pseudodifferential operator $\Psi(x,D_x)$, elliptic in $\omega_1$ such that 
$$ \Psi(x,D_x)u \in L^2_{loc}(]0,T[,H^{r+\rho}) $$
\end{prop}

\begin{corollaire}
\label{corollairepropagreg}
With the notations of the Proposition, if an open set $\omega$ satisfies Assumption \ref{geometriccontrol} and $a(x) u\in L^2_{loc}(]0,T[,H^{r+\rho})$, with $a \in C^{\infty}(M)$, as in  (\ref{lienoma}), then $u\in L^2_{loc}(]0,T[,H^{r+\rho}(M)$.
\end{corollaire}

\Dem We first regularize : $u_n=e^{\frac{1}{n}\Delta }u$ with $\left\|u_n\right\|_{X^{r,b}_T}\leq C $. Set $s=r+\rho$\\
Let $B(x,D_x)$ be a pseudodifferential operator of order $2s-1=2r+2\rho-1$, that will be chosen later and $A=A(t,x,D_x)=\varphi(t)B(x,D_x)$ where $\varphi \in C^{\infty}_0(]0,T[)$.\\ 
If $L=i\partial_t +\Delta$, we write 
\begin{eqnarray*}
& &(L u_n,A^*u_n)_{L^2(]0,T[\times M)}-(Au_n,Lu_n,)_{L^2(]0,T[\times M)}\\
&=& ([A,\Delta]u_n,u_n)_{L^2(]0,T[\times M)}-(i\varphi ' B u_n,u_n)_{L^2(]0,T[\times M)}
\end{eqnarray*}
\begin{eqnarray*}
|( Au_n,f_n)_{L^2(]0,T[\times M)}| &\leq & \| Au_n\|_{X^{-r,b}_T}\|f_n\|_{X^{r,-b}_T}\\
&\leq & \| u_n\|_{X^{r+2\rho-1+b,b}_T}\|f_n\|_{X^{r,-b}_T}
\end{eqnarray*}
As we have chosen $ \rho\leq \frac{1-b}{2}$, we have $r+2\rho-1+b \leq r$ and so 
\begin{eqnarray*}
|(Au_n,f_n)_{L^2(]0,T[\times M)}| & \leq & C \| u_n\|_{X^{r,b}_T}\|f_n\|_{X^{r,-b}_T}\leq C.
\end{eqnarray*}
Similarly 
\begin{eqnarray*}
|( \varphi' B u_n,u_n)_{L^2(]0,T[\times M)}| &\leq & C \| u_n\|_{X^{r,b}_T}\|u_n\|_{X^{r,-b}_T} \leq C
\end{eqnarray*}
Then, 
$$([A,\Delta]u_n,u_n)_{L^2(]0,T[\times M)}=\int_0^T \varphi(t) ([B,\Delta]u_n(t),u_n(t))_{L^2(M)} dt $$
is uniformly bounded. Then, we select $B$ by means of symplectic geometry. Take $\omega_1 \in \Gamma_{\omega_0}$, $U$ and $V$ two small conical neighborhoods, respectively of $\omega_1$ and $\omega_0$. For every symbol $c(x,\xi)$, of order $s$, supported in $U$, one can find a symbol $b(x,\xi)$ of order $2s-1$ such that
$$\frac{1}{i}\left\{\sigma_2(\Delta),b(x,\xi)\right\} = \left|c(x,\xi)\right|^2+r(x,\xi)$$
with $r(x,\xi)$ of order $2s$ and supported in $V$. We take $B$ a pseudodifferential operator with principal symbol $b$ so that $[B,\Delta]$ is a pseudodifferential operator of principal symbol $\left|c(x,\xi)\right|^2+r(x,\xi)$. Then, if we choose $c(x,\xi)$ elliptic at $\omega_1$, we conclude
$$\intt \varphi(t) \left\|c(x,D_x)u_n(t,x)\right\|^2_{L^2} dt \leq C.$$ 
This ends the proof of Proposition \ref{propreg}.
\begin{corollaire}
\label{corpropagreg}
Here dim $M \leq 3$ and $b>1/2$.
Let $u\in X^{1,b}_T$ solution of 
\begin{eqnarray}
\label{eqnpropagreg}
\left\lbrace
\begin{array}{rcl}
i\partial_t u + \Delta u &=& |u|^2u+u\textnormal{ on } [0,T]\times M\\
\partial_t u&=&0 \textnormal{ on } ]0,T[\times \omega \\
\end{array}
\right.
\end{eqnarray}
where $\omega$ satisfies Assumption \ref{geometriccontrol}.\\
Then $u\in C^{\infty}(]0,T[\times M)$.
\end{corollaire}

\Dem We have $u\in L^{\infty}([0,T],H^1)$, and so in $L^{\infty}([0,T],L^6)$ by Sobolev embedding. Then, we infer that $|u|^2 u \in L^{\infty}([0,T],L^2(M))$.\\ 
On $]0,T[\times \omega $, we have
$$\Delta u = \left|u\right|^2 u +u.$$ 
Therefore, $\Delta u \in L^2([0,T],L^2(\omega))$ and $u\in L^2(]0,T[\times H^2(\omega))$. Since $H^2(\omega)$ is an algebra, we can go on the same reasonning to conclude that $u\in C^{\infty}(]0,T[\times \omega)$.\\
By applying once Corollary \ref{corollairepropagreg}, we get $u\in L^2_{loc}([0,T],H^{1+\frac{1-b}{2}})$. Then we can pick $t_0$ such that $u(t_0) \in H^{1+\frac{1-b}{2}}$. We can then solve in $X^{1+\frac{1-b}{2},b}$ our nonlinear Schrödinger equation with initial data $u(t_0)$. By uniqueness in $X^{1,b}_T$, we can conclude that $u\in X^{1+\frac{1-b}{2},b}_T$.\\
By iteration, we get that $u\in L^2(]0,T[,H^{r})$ for every $r\in \R$ and $u\in C^{\infty}([0,T],M)$.
\begin{corollaire}
\label{corprolunique}
If, in addition to Corollary \ref{corpropagreg}, $\omega$ satisfies Assumption \ref{uniquecontinuation}, then $u=0$.
\end{corollaire}
\begin{proof}
Using Corollary \ref{corpropagreg}, we infer that $u\in C^{\infty}(]0,T[\times M)$.\\
Taking time derivative of equation (\ref{eqnpropagreg}), $v=\partial_t u$ satisfies
\begin{eqnarray}
\left\lbrace
\begin{array}{rcl}
i\partial_t v + \Delta v +f_1~ v+ f_2~ \bar{v}&=&0\\
v&=&0 \textnormal{     on   } ]0,T[\times \omega
\end{array}
\right.
\end{eqnarray} 
 for some $f_1$, $f_2 \in C^{\infty}(]0,T[\times M)$. Assumption \ref{uniquecontinuation} gives $v=\partial_t u=0$. 
Multiplying (\ref{eqnpropagreg}) by $\bar{u}$ and integrating , we get
$$\int_M \left|\nabla u \right|^2 +\int_M \left|u\right|^4+\int_M \left|u \right|^2=0$$
and so $u=0$.
\end{proof}
\begin{remarque}
\label{corproluniquelin}
We have the same conclusion for $u\in X^{1,b}_T$ solution of 
\begin{eqnarray}
\left\lbrace
\begin{array}{rcl}
i\partial_t u + \Delta u &=& u\textnormal{ on } [0,T]\times M\\
\partial_t u&=&0 \textnormal{ on } ]0,T[\times \omega \\
\end{array}
\right.
\end{eqnarray}
\end{remarque}
\section{Stabilization}
Theorem \ref{thmstab} is a consequence of the following Proposition
\begin{prop}
\label{propstabilisation}Let $a \in C^{\infty}(M)$, as in  (\ref{lienoma}).
Under Hypothesis  \ref{geometriccontrol}, \ref{uniquecontinuation} and $\ref{hypinegXsb}$, for every $T>0$ and every $R_0>0$, there exists a constant $C>0$ such that inequality
$$E(0)\leq C \int_{0}^{T}\left\|(1-\Delta)^{-1/2}a(x)\partial_t u\right\|^2_{L^2} dt$$
holds for every solution $u$ of the damped equation 
\begin{eqnarray}
\label{eqndamped}
\left\lbrace
\begin{array}{rcl}
i\partial_t u + \Delta u -(1+|u|^2)u&=&a(x)(1-\Delta)^{-1}a(x)\partial_t u \textnormal{ on } [0,T]\times M\\
u(0)&=&u_{0} \in H^1
\end{array}
\right.
\end{eqnarray}
 and $\left\|u_0\right\|_{H^1}\leq R_0$.
\end{prop}
\begin{proof}[Proof of Proposition \ref{propstabilisation} $\Rightarrow$ Theorem \ref{thmstab}] For any $f \in H^1(M)$, Sobolev embeddings yield
\bna
 E(f)\leq C \left(\left\|f\right\|^2_{H^1} +\left\|f\right\|^4_{H^1}\right) \\
\left\|f\right\|_{H^1} \leq C\left(E(f)\right)^{1/2} .
\ena
As the energy is decreasing, if $\left\|u_0\right\|_{H^1}\leq R_0$, we can find another $\widetilde{R}_0$ such that $\left\|u(t)\right\|_{H^1}\leq \widetilde{R}_0$ for any $t>0$. For this range of values, we have 
\bnan
\label{equivenerg}
C^{-1}\left(E(f)\right)^{1/2}\leq \left\|f\right\|_{H^1} \leq C\left(E(f)\right)^{1/2}
\enan
 for one $C>0$ depending on $R_0$.

 We apply Proposition \ref{propstabilisation} with this bound and obtain $E(t)\leq Ce^{-\gamma(R_0) t} E(0)$. Then, for $t>t(R_0)$, we have $\left\|u(t)\right\|_{H^1}\leq 1$. 

We take $\gamma(1)$ the decay rate corresponding to the bound $1$. Therefore, for $t>t(R_0)$, we get $\left\|u(t)\right\|_{H^1}\leq Ce^{-\gamma(1) (t-t(R_0))}\left\|u(t(R_0))\right\|_{H^1}$. This yields a decay rate independant of $R_0$ as announced, while the coefficient $C$ may strongly depend on $R_0$.
\end{proof}
\begin{remarque}
If we make the change of unknown $w=e^{-it}u$, $w$ is solution of the new damped equation
\begin{eqnarray*}
\left\lbrace
\begin{array}{rcl}
i\partial_t w + \Delta w -|w|^2w&=&a(x)(1-\Delta)^{-1}a(x)(\partial_tw-iw) \textnormal{ on } [0,T]\times M\\
w(0)&=&u_{0} \in H^1
\end{array}
\right.
\end{eqnarray*}
This modification is necessary because there is not exponential decay for the damped equation (\ref{eqndamped}) with $|u|^2u$ instead of $(1+|u|^2)u$. We check for example that for $a=1$, the solution with constant Cauchy data $u_0$ satisfies
$$|u(t)|^2=\frac{|u_0|}{1+|u_0|t}.$$
Moreover, it also proves that the solution is global in time only on $\R^{+}$ (this restriction remains with the non linearity $(1+|u|^2)u$).
\end{remarque}

\begin{proof}[Proof of Proposition \ref{propstabilisation}]
We argue by contradiction, we suppose the existence of a sequence $(u_n)$ of solutions of (\ref{eqndamped}) such that
$$\left\|u_n(0)\right\|_{H^1}\leq R_0$$
and
\begin{eqnarray}
\label{majorationdamped}
 \int_{0}^{T}\left\|(1-\Delta)^{-1/2}a(x)\partial_t u_n\right\|^2_{L^2} dt \leq \frac{1}{n} 
 E(u_n(0))
 \end{eqnarray}
We note $\alpha_n=E(u_n(0))^{1/2}$.
By the Sobolev embedding for the $L^4$ norm, we have $\alpha_n \leq C (R_0)$. So, up to extraction, we can suppose that $\alpha_n \longrightarrow \alpha$.\\
We will distinguich two cases : $\alpha>0$ and $\alpha=0$.

\bigskip

First case : $\alpha_n \longrightarrow \alpha>0$\\
By decreasing of the energy, $(u_n)$ is bounded in $L^{\infty}([0,T],H^1)$ and so in $\Xubt$. Then, as $\Xubt$ is a separable Hilbert we can extract a subsequence such that $ u_n\rightharpoonup u$ weakly in $\Xubt$ ans strongly in $X^{s,b'}_T$ for one $u \in \Xubt$ and $s>s_0$. Therefore, $|u_n|^2u_n$ converges to $|u|^2u$ in $X^{s,-b'}_T$.

Using (\ref{majorationdamped}) and passing to the limit in the equation verified by $u_n$, we get
 \begin{eqnarray*}
\left\lbrace
\begin{array}{rcl}
i\partial_t u + \Delta u &=&|u|^2u+u \textnormal{ on } [0,T]\times M\\
\partial_t u&=&0 \textnormal{ on } ]0,T[\times \omega \\
\end{array}
\right.
\end{eqnarray*}
Using Corollary \ref{corprolunique}, we infer $u=0$.\\
Therefore, we have, up to new extraction, $u_n(0)\rightharpoonup 0$ in $H^1$. Using Proposition \ref{proplinearisation} of linearisation, we infer that $|u_n|^2u_n \rightarrow 0$ in $X^{1,-b'}_T$.\\
Moreover, because of (\ref{majorationdamped}) we have
$$a(x)(1-\Delta)^{-1}a(x)\partial_t u_n \underset{L^2([0,T],H^1)}{\longrightarrow}0$$
and the convergence is also in $X^{1,-b'}_T$.\\
Then,  estimate (\ref{majorationdamped}) also implies
$a(x)\partial_t u_n \underset{L^2([0,T],H^{-1})}{\longrightarrow}0$.\\
Using equation (\ref{eqndamped}), we obtain
$$a(x)\left[\Delta u_n -u_n-|u_n|^2u_n-a(x)(1-\Delta)^{-1}a(x)\partial_t u_n\right] \underset{L^2([0,T],H^{-1})}{\longrightarrow}0$$
By Sobolev embedding, $u_n$ tends to $0$ in $L^{\infty}([0,T],L^p)$ for any $p<6$. Therefore, $|u_n|^2u_n$ converges to $0$ in $L^{\infty}([0,T],L^q)$ for $q<2$ and so in $L^2([0,T],H^{-1})$. Thus, we get
$$a(x)(\Delta -1)u_n \underset{L^2([0,T],H^{-1})}{\longrightarrow}0.$$
Therefore, $(1-\Delta)^{1/2} a(x) u_n = (1-\Delta)^{-1/2}a(x)(1-\Delta)u_n+(1-\Delta)^{-1/2}[(1-\Delta),a(x)]u_n$ converges to $0$ in $L^2([0,T],L^2)$.

In conclusion, we have
\bna
\left\lbrace
\begin{array}{l}
u_n \rightharpoonup 0 \quad \textnormal {in} \quad X^{1,b'}_T\\
a(x)u_n \rightarrow 0 \quad \textnormal {in} \quad L^2([0,T],H^{1})\\
i\partial_t u_n +\Delta u_n -u_n \longrightarrow 0 \quad \textnormal {in} \quad X^{1,-b'}_T
\end{array}
\right.
\ena
Thus, changing $u_n$ into $e^{it}u_n$ and using that the multiplication by $e^{it}$ is continuous on any $\Xsbt$ (see Lemma \ref{lemmetps}), we are in position to apply Corollary \ref{corrpropag2}. Hence, as we have $1-b'>1/2$, it yields 
$$u_n(0) \underset{H^1}{\longrightarrow}0. $$
In particular, $E(u_n(0))\rightarrow 0$ which is a contradiction to our hypothesis $\alpha >0$.

\bigskip

Second case : $\alpha_n \longrightarrow 0$\\
Let us make the change of unknown $v_n=u_n/\alpha_n$. $v_n$ is solution of the system
$$i\partial_t v_n +\Delta v_n -a(x)(1-\Delta)^{-1}a(x)\partial_t v_n= v_n+\alpha_n^2|v_n|^2v_n$$
and
\begin{eqnarray}
\label{majorenergystab}
\int_{0}^{T}\left\|(1-\Delta)^{-1/2}a(x)\partial_t v_n\right\|^2_{L^2} dt \leq \frac{1}{n} 
\end{eqnarray}
For a constant depending on $R_0$, we still have (\ref{equivenerg}). Therefore, we write
$$ \left\|v_n(t)\right\|_{H^1}=\frac{ \left\|u_n(t)\right\|_{H^1} }{ E(u_n(0))^{1/2} }\leq C\frac{E(u_n(t))^{1/2}  }{E(u_n(0))^{1/2}}\leq C$$
\begin{eqnarray}
\label{H1tend1}
\left\|v_n(0)\right\|_{H^1}=\frac{ \left\|u_n(0)\right\|_{H^1} }{ E(u_n(0))^{1/2} }\geq C>0
\end{eqnarray}
Thus, we have $\left\|v_n(0)\right\|_{H^1}\approx 1$ and $v_n$ is bounded in $L^{\infty}([0,T],H^1)$.\\
By the same estimates we made in the proof of Proposition \ref{thmexistenceNl}, we obtain
$$\left\|v_n\right\|_{X^{1,b}_T} \leq C \left\|v_n(0)\right\|_{H^1}+CT^{1-b-b'}\left(\left\|v_n\right\|_{X^{1,b}_T}+\alpha_n^2 \left\|v_n\right\|_{X^{1,b}_T}^3\right) $$
Then, if we take $CT^{1-b-b'}<1/2$, independant of $v_n$, we have
$$\left\|v_n\right\|_{X^{1,b}_T} \leq C(1+\alpha_n^2 \left\|v_n\right\|_{X^{1,b}_T}^3).$$
By a boot strap argument, we conclude that, $\left\|v_n\right\|_{X^{1,b}_{T}}$ is uniformly bounded. Using Lemma \ref{lemmerecouvrement}, we conclude that it is bounded on $\Xubt$ for some large $T$ and then, $\alpha_n^2|v_n|^2v_n $ tends to $0$ in $X^{1,-b'}_T$.\\
Then, we can extract a subsequence such that $v_n \rightharpoonup v$ in $\Xubt$ where $v$ is solution of 
\begin{eqnarray*}
\left\lbrace
\begin{array}{rcl}
i\partial_t v + \Delta v &=&v \textnormal{ on } [0,T]\times M\\
\partial_t v&=&0 \textnormal{ on } ]0,T[\times \omega \\
\end{array}
\right.
\end{eqnarray*}
It implies $v=0$ by Remark \ref{corproluniquelin}.\\
Estimate (\ref{majorenergystab}) yields that $a(x)(1-\Delta)^{-1}a(x)\partial_t v_n $ converges to $0$ in $L^2([0,T],H^1)$ and so in $X^{1,-b'}_T$.

We finish the proof as in the first case to conclude the convergence of $v_n$ to $0$ in $\Xubt$. This contradicts (\ref{H1tend1}).
\end{proof}
\section{Controllability of the linear equation}
\label{sectioncontrollin}
\subsection{Observability estimate}
\begin{prop}
\label{propobserv}
Assume that $(M,\omega)$ satisfies Hypothesis \ref{geometriccontrol}, \ref{hypinegXsb} and \ref{uniquecontinuationH1}. Let $a \in C^{\infty}(M)$, as in  (\ref{lienoma}), taking real values. Then, for every $-1\leq s\leq 1$, $T>0$ and $A>0$, there exists $C$ such that estimate 
$$\left\|u_0\right\|^2_{H^{s}}\leq C \intT \left\|a u(t)\right\|^2_{H^{s}}dt $$
holds for every solution $u(t,x)\in X^{s,b}_T$ of the system
\begin{eqnarray}
\label{eqlin1}
\left\lbrace
\begin{array}{rcl}
i\partial_t u + \Delta u &=&\pm2|w|^2 u \pm w^2 \bar{u}\quad\textnormal{on}\quad[0,T]\times M\\
u(0)&=&u_{0} ~\in H^{s}
\end{array}
\right.
\end{eqnarray}
with one $w$ satisfying $\left\|w\right\|_{\Xubt}\leq A$.
\end{prop}
\begin{proof}
We only treat the case with $2|w|^2 u + w^2 \bar{u}$. The others are similar. We argue by contradiction. Let $u_n \in X^{s,b}_T$ be a sequence of solution to (\ref{eqlin1}) with some associated $w_n$ such that
\bnan
\label{hypocontradict}
\left\| u_n(0) \right\|_{H^{s}}=1,\quad \intT \left\|a u_n \right\|^2_{H^{s}} \rightarrow 0
\enan
and
$$ \left\|w_n\right\|_{\Xubt}\leq A.$$
Proposition \ref{thmlineareqn} of existence yields that $u_n$ is bounded in $\Xsbt$ and we can extract a subsequence such that $u_n$ converges strongly in $X^{s-1+b,-b}_T $ to some $u\in \Xsbt$ ($b<1-b'<1$).\\
Then, using Lemma \ref{lmtrilinXsbbis}, we infer that $2|w_n|^2 u_n + w_n^2 \bar{u_n}$ is bounded in $X^{s,-b'}_T$. We can extract another subsequence such that it converges strongly in $X^{s-1+b,-b}_T $ (here we use $-b<-1/2<-b'$) to some $ \Psi \in X^{s,-b'}_T$. \\
Denoting $r_n=u_n-u$ and $f_n=2|w_n|^2 u_n +w_n^2 \bar{u_n}-\Psi$, we can apply Proposition \ref{thmpropagation2} of propagation of compactness. As $\omega$ satisfies geometric control and $au_n \rightarrow 0$ in $L^2([0,T],H^{s})$, we obtain that $r_n \rightarrow 0$ in $L^2_{loc}([0,T],H^{s})$.\\
$r_n$ is also bounded in $\Xsbt$ and we deduce, by interpolation, that $r_n$ tends to $0$ in $X^{s,b'}_I$ for every $I\subset\subset ]0,T[$.
  
Now, we want to prove that $u\equiv 0$ using unique continuation. As $w_n$ is bounded in $\Xubt$, we can extract a subsequence such that it converges weakly to some $w\in \Xubt$. We have to prove that $u$ is solution of a linear Schrödinger equation with potential. But the fact that $|w_n|^2u_n$ converges weakly to $|w|^2u$ is not guaranteed and actually uses the fact that the regularity $H^1$ is subcritical (see the article of L. Molinet \cite{Molinet} where the limit of the product is not the expected one).\\
We decompose
\bna
u_n|w_n|^2-u|w|^2&=&(u_n-u)|w_n|^2+u\left[|w_n-w|^2-w(\overline{w-w_n})-\overline{w}(w-w_n)\right]\\
&=&1+2+3+4
\ena
Term 1 converges strongly to $0$ in $X^{s,-b'}_T$ because $u_n-u$ tends to $0$ in $X^{s,b'}_T$ and $w_n$ is bounded in $\Xubt$. For term 2, we use tame estimate for $\varepsilon$ such that $1-\varepsilon>s_0$
 $$\left\|u|w_n-w|^2\right\|_{X^{s,-b'}_T}\leq \left\|u\right\|_{\Xsbt}\left\|w_n-w\right\|_{X^{1-\varepsilon,b'}_T}\left\|w_n-w\right\|_{X^{1,b'}_T}.$$
By compact embedding, $w_n-w$ converges, up to extraction, strongly to $0$ in $X^{1-\varepsilon,b'}_T$ and Term 2 converges strongly in $ X^{s,-b'}_T$. Terms 3 and 4 converge weakly to $0$ in $X^{-1,-b}_T$ and so in the distributional sense.

Finally, we conclude that the limit of $u_n|w_n|^2$ is $u|w|^2$. We obtain similarly that $w_n^2 \overline{u_n}$ converges in the distributional sense to $w^2 \bar{u}$. Therefore, $u$ is solution of
\begin{eqnarray*}
\left\lbrace
\begin{array}{c}
i\partial_t u + \Delta u =2|w|^2 u +w^2 \bar{u}\\
u=0 \textnormal{ on } [0,T]\times \omega
\end{array}
\right.
\end{eqnarray*}
Using Corollary \ref{corollairepropagreg}, we infer that $u\in L^2_{loc}([0,T],H^{s+\frac{1-b}{2}})$ and existence Proposition \ref{thmlineareqn} yields that it actually belongs to $X^{s+\frac{1-b}{2},b}_T$. By iteration, we obtain that $u\in X^{1,b}_T$. Then, we can apply Assumption \ref{uniquecontinuationH1} and we have in fact $u=0$.

We pick $t_0 \in [0,T]$ such that $u_n(t_0)$ converges strongly to $0$ in $H^s$. Estimate (\ref{inegeqnlin}) of existence Proposition \ref{thmlineareqn} yields strong convergence to $0$ of $u_n$ in $\Xsbt$. Therefore, $\left\|u_n(0)\right\|_{H^s}$ tends to $0$, which contradicts (\ref{hypocontradict}).   
\end{proof}
\subsection{Linear control}
\begin{prop}
\label{proplincontrol}
Assume that $(M,\omega)$ satisfies Hypothesis \ref{geometriccontrol}, \ref{hypinegXsb} and \ref{uniquecontinuationH1}. Let $-1\leq s\leq 1$, $T>0$ and $w\in X^{1,b}_T$. For every $u_0\in H^s(M)$ 
there exists a control $g\in C([0,T],H^s)$ supported in $[0,T]\times \overline{\omega}$, such that the unique solution $u$ in $X^{s,b}_T$ of the Cauchy problem
\begin{eqnarray}
\left\lbrace
\begin{array}{rcl}
i\partial_t u + \Delta u &=&\pm2|w|^2 u  \pm w^2 \overline{u} +g\quad \textnormal{on}\quad[0,T]\times M\\
u(0)&=&u_{0} \in H^s(M)
\end{array}
\right.
\end{eqnarray}
satisfies $u(T)=0$.
\end{prop}
\begin{proof}
We only treat the case with $2|w|^2 u + w^2 \bar{u}$. Let $a(x)\in C^{\infty}(M)$ real valued, as in (\ref{lienoma}). We apply the HUM method of J.L. Lions. We consider the system
\begin{eqnarray*}
\left\lbrace
\begin{array}{rcll}
i\partial_t u + \Delta u &=& 2|w|^2 u + w^2 \overline{u}+g&g\in L^2([0,T],H^{s}) \quad u(T)=0\\
i\partial_t v + \Delta v &=& 2|w|^2 v - w^2 \overline{v}&v(0)=v_0\in H^{-s}
\end{array}
\right.
\end{eqnarray*}
These equations are well posed in $X^{s,b}_T$ and $X^{-s,b}_T$ thanks to Proposition \ref{thmlineareqn}.
The equation verified by $v$ is the dual as the one of $u$ for the real duality (the equation is not $\C$ linear). Then, multiplying the first system by $i\overline{v}$, integrating and taking real part, we get (the computation is true for $w$, $g$ and $v_0$ smooth, we extend it by approximation)
\bna
\Re (u_0,v_0)_{L^2}=\Re \int_0^T (ig,v)_{L^2} dt 
\ena
where $(\cdot,\cdot)_{L^2}$ is the complex duality on $L^2(M)$. We define the continuous map $S : H^{-s} \rightarrow H^s$ by $Sv_0=u_0$ with the choice
$$g=Av=-ia(x)(1-\Delta)^{-s}a(x).$$
This yields
$$\Re (Sv_0,v_0)_{L^2}= \Re \int_{0}^T (a(x)(1-\Delta)^{-s}a(x)v,v)=\int_0^T \norL{1-\Delta)^{-s/2}a(x)v}^2= \int_0^T \nor{a(x)v}{H^{-s}}^2$$
Thus, $S$ is self-adjoint and positive-definite thanks to observability estimate of Proposition \ref{propobserv}. It therefore defines an isomorphism from $H^{-s}$ into $H^s$. Moreover, we notice that the norms of $S$ and $S^{-1}$ are uniformly bounded as $w$ is bounded in $\Xubt$.  
\end{proof}
\begin{prop}
Assume $0\leq s\leq 1$, $w=0$ and $(M,\omega)$ is either :\\
-($\Tot$,any open set)\\
-($S^2\times S^1$,(a neighborhood of the equator)$\times S^1$)\\
-($S^2\times S^1$,$S^2\times $ (any open set of $S^1$))\\
Then, the same conclusion as Proposition \ref{proplincontrol} holds.
\end{prop}
\begin{proof}
By following the proof of Proposition \ref{proplincontrol}, we are reduced to proving an observability estimate
$$\left\|u_0\right\|^2_{H^{-s}}\leq C \intT \left\|a(x) e^{it\Delta}u_0\right\|^2_{H^{-s}}dt$$
These results are already known for $s=0$ : \\
-for $\Tot$, this was first proved by S. Jaffard \cite{Jaffard} in dimension $2$ and generalized to any dimension by V. Komornik \cite{Komornik}.\\
-the others example are of the form $(M_1\times M_2,\omega_1\times M_2)$ were $\omega_1$ satisfies observability estimate.

We can extend them to any $s$, with $0\leq s\leq 1$ by writing $\nor{u_0}{H^{-s}}=\nor{(1-\Delta)^{-s/2}u_0}{L^2}$. We conclude using observablility estimate in $L^2$ and commutator estimates.\\ 
Actually, Proposition \ref{thmpropagrecontrol} of the next section proves that controllability in $L^2$ implies controllability in $H^s$, $0\leq s\leq 1$, with the HUM operator constructed on $L^2$. This yields the observability estimate in $H^{-s}$ and for that reason, we do not detail the previous argument. 
\end{proof}
\subsection{Regularity of the control}
\label{subsectprop}
This section is strongly inspired by the work of B. Dehman and G. Lebeau \cite{HumDehLeb}. It express the fact that the HUM operator constructed on a space $H^s$ propagates some better regularity. We extend this result to the Schrödinger equation with some rough potentials.
  
Let $T>0$, $s\in [-1,1]$ and $w\in X^{1,b}_T$. As in the the proof of Proposition \ref{proplincontrol}, we denote $S= S_{s,T,w,a} : H^{-s}\rightarrow H^s$ the HUM operator of control associated to the trajectory $w$ by $S\Phi_0=u_0$ where
\begin{eqnarray*}
\left\lbrace
\begin{array}{rcl}
i\partial_t \Phi + \Delta \Phi &=& 2|w|^2 \Phi - w^2 \overline{\Phi}\\
\Phi(x,0)&=&\Phi_{0}(x)\in H^{-s}
\end{array}
\right.
\end{eqnarray*}
 and $u$ solution of 
\begin{eqnarray*}
\left\lbrace
\begin{array}{rcl}
i\partial_t u + \Delta u &=&2|w|^2 u + w^2 \overline{u}+A \Phi\\
u(T)&=&0
\end{array}
\right.
\end{eqnarray*}
where $A=-ia(x)(1-\Delta)^{-s}a(x)$.
\begin{prop}
\label{thmpropagrecontrol}
Suppose Assumptions \ref{hypinegXsb} and \ref{hypcommut} are fulfilled. Let $0\leq s_0<s \leq 1$, $\varepsilon=1-s$ and $w\in \Xubt$. Denote $S=S_{s,T,w,a}$ the operator defined above. We assume that $S$ is an isomorphism from $H^{-s}$ into $H^s$. Then, $S$ is also an isomorphism from $H^{-s+\varepsilon}$ into $H^{s+\varepsilon}=H^1$.
\end{prop}
\begin{proof}
First, we show that $S$ maps $H^{-s+\varepsilon}$ into $H^{s+\varepsilon}$.

Let $\Phi_0\in H^{-s+\varepsilon}$. By existence Proposition \ref{thmlineareqn}, we have $\Phi \in X^{-s+\varepsilon,b}_T$, then $A\Phi \in L^2([0,T],H^{s+\varepsilon})$ and existence Proposition \ref{thmlineareqn} gives again $u\in X^{s+\varepsilon,b}_T$ and $u(0)=S\Phi_0\in H^{s+\varepsilon}$.

To finish, we only have to prove that $S\Phi_0=u_0\in H^{s+\varepsilon}$ implies $\Phi_0\in H^{-s+\varepsilon}$. As we already know that $\Phi_0\in H^{-s}$, we need to prove that $(-\Delta)^{\varepsilon/2}\Phi_0\in H^{-s}$. We use the fact that $S$ is an isomorphism from $H^{-s}$ into $H^{s}$. Denote $D^{\varepsilon}=(-\Delta)^{\varepsilon/2}$.
\bna
\left\|D^{\varepsilon}\Phi_0\right\|_{H^{-s}}\leq C\left\|S D^{\varepsilon}\Phi_0\right\|_{H^{s}}& \leq &C\left\|SD^{\varepsilon}\Phi_0-D^{\varepsilon}S\Phi_0\right\|_{H^{s}}+C\left\|D^{\varepsilon}S\Phi_0\right\|_{H^{s}}\\
&\leq &C\left\|SD^{\varepsilon}\Phi_0-D^{\varepsilon}u_0\right\|_{H^{s}}+C\left\|u_0\right\|_{H^{s+\varepsilon}}
\ena
Let $\varphi$ solution of
 \begin{eqnarray*}
\left\lbrace
\begin{array}{rcl}
i\partial_t \varphi + \Delta \varphi &=& 2|w|^2 \varphi - w^2 \overline{\varphi}\\
\varphi(x,0)&=&D^{\varepsilon}\Phi_{0}(x)
\end{array}
\right.
\end{eqnarray*} 
 and $v$ solution of 
\begin{eqnarray*}
\left\lbrace
\begin{array}{rcl}
i\partial_t v + \Delta v &=&2|w|^2 v + w^2 \overline{v}+A \varphi\\
v(T)&=&0
\end{array}
\right.
\end{eqnarray*}
So that $v(0)=SD^{\varepsilon}\Phi_0$. We need to estimate $\nor{v(0)-D^{\varepsilon}u_0}{H^s}$. But $r=v-D^{\varepsilon} u$  is solution of
 \begin{eqnarray*}
\left\lbrace
\begin{array}{rcl}
i\partial_t r + \Delta r &=&2|w|^2 r + w^2 \overline{r}- 2 [D^{\varepsilon},|w|^2]u-[D^{\varepsilon},w^2]\overline{u}+A (\varphi-D^{\varepsilon}\Phi)-[D^{\varepsilon},A]\Phi\\
r(T)&=&0
\end{array}
\right.
\end{eqnarray*}
Then, using Proposition \ref{thmlineareqn} we obtain
\bna
\left\|r_0\right\|_{H^{s}}\leq C \left\|r\right\|_{X^{s,b}_T}&\leq &C\left(\left\| [D^{\varepsilon},|w|^2]u\right\|_{X^{s,-b'}_T}+\nor{[D^{\varepsilon},w^2]}{X^{s,-b'}_T}\right.\\
& &\left.+ \left\|A (\varphi-D^{\varepsilon}\Phi)\right\|_{X^{s,-b'}_T}+\left\|[D^{\varepsilon},A]\Phi\right\|_{X^{s,-b'}_T}\right)
\ena
Lemma \ref{quadrdoneXsb} of the Appendix, Section \ref{sectcommut} gives us some estimates about the commutators. For the last term, we notice that $[D^{\varepsilon},A]$ is a pseudodifferential operator of order $\varepsilon-2s-1\leq -2s$.
\bna
\left\|r_0\right\|_{H^{s}}\leq C\left(\left\|w\right\|^2_{X^{s+\varepsilon,b'}_T}\left\|u\right\|_{X^{s,b'}_T}+\left\|A (\varphi-D^{\varepsilon}\Phi)\right\|_{X^{s,-b'}_T}+\left\|\Phi\right\|_{L^2([0,T],H^{-s})}\right)
\ena
We already know that $u\in X^{s,b'}_T$, $w \in X^{s+\varepsilon,b'}_T$ and $\Phi \in X^{-s,b}_T$. We only have to estimate \\
$\left\|A (\varphi-D^{\varepsilon}\Phi)\right\|_{X^{s,-b'}_T}\leq C \left\|\varphi-D^{\varepsilon}\Phi\right\|_{L^2([0,T],H^{-s})}$. But $d=\varphi-D^{\varepsilon}\Phi$ is solution of
\begin{eqnarray*}
\left\lbrace
\begin{array}{rcl}
i\partial_t d + \Delta d &=& 2|w|^2 d - w^2 \overline{d}-2[D^{\varepsilon},|w|^2]\Phi+[D^{\varepsilon},w^2]\overline{\Phi}\\
d(x,0)&=&0
\end{array}
\right.
\end{eqnarray*}
Thus, using Proposition \ref{thmlineareqn}, we get 
 \bna
 \left\|\varphi-D^{\varepsilon}\Phi\right\|_{L^2([0,T],H^{-s})}\leq C \left\|d\right\|_{X^{-s,b}_T}\leq C\left( \left\|[D^{\varepsilon},|w|^2]\Phi\right\|_{X^{-s,-b'}_T}+\left\|[D^{\varepsilon},w^2]\overline{\Phi}\right\|_{X^{-s,-b'}_T}\right)
 \ena
The second part of Lemma \ref{quadrdoneXsb} of the Appendix allows us to conclude.
\end{proof}
\section{Control near a trajectory}
Theorem \ref{thmcontrolenonlin} and \ref{thmcontrolenonlinw0} are consequences of the following Proposition
\begin{prop}
\label{propcontrotraj}
Suppose Assumptions \ref{hypinegXsb} and \ref{hypcommut} are fulfilled. Let $T>0$ and $w\in \Xubt$ a controlled trajectory, i.e. solution  of 
\begin{eqnarray*}
i\partial_t w + \Delta w &=&\pm|w|^2w+g_1 \textnormal{ on } [0,T]\times M
\end{eqnarray*}
with $g_1\in L^2([0,T],H^1(M))$, supported in $\overline{\omega}$.
Let $1\geq s>s_0\geq 0$ .
Assume that the HUM operator $S=S_{s,T,w,a}$, defined in Subsection \ref{subsectprop}, is an isomorphism from $H^{-s}$ into $H^s$.\\
There exists $\varepsilon>0$ such that for every $u_0\in H^{s}$ with $\left\|u_0-w(0)\right\|_{H^s}<\varepsilon$, there exists $g\in C([0,T],H^s)$ supported in $[0,T]\times \overline{\omega} $ such that the unique solution $u$ in $\Xsbt$ of
\begin{eqnarray}
\label{eqnonlinsource}
\left\lbrace
\begin{array}{rcl}
i\partial_t u + \Delta u &=&\pm|u|^2u+g\\
u(x,0)&=&u_{0}(x)
\end{array}
\right.
\end{eqnarray}
fulfills $u(T)=w(T)$.\\
Moreover, we can find another $\varepsilon>0$ depending only on $T$,$s$, $\omega$ and $\left\|w\right\|_{\Xubt}$ such that for any $u_0\in H^1$ with $\left\|u_0-w(0)\right\|_{H^s}<\varepsilon$, the same conclusion holds with $g\in C([0,T],H^1)$.
\end{prop}
\begin{proof}
In the demonstration, we denote $C$ some constants that could actually depend on $T$, $\left\|w\right\|_{\Xubt}$ and $s$. The final $\varepsilon$ will have the same dependence. We make the proof for the defocusing case, but since there is no energy estimate, it is the same in the other situation.

We linearize the equation as in Proposition \ref{thmexistproche}. If $u=w+r$, then $r$ is solution of 
\begin{eqnarray*}
\left\lbrace
\begin{array}{rcl}
i\partial_t r + \Delta r &=&2\left|w\right|^2r+w^2\bar{r}+F(w,r)+g-g_1\\
r(x,0)&=&r_{0}(x)
\end{array}
\right.
\end{eqnarray*}
with $F(w,r)=2\left|r\right|^2w+r^2\bar{w} +\left|r\right|^2r$. We seek $g$ under the form $g_1+A\Phi$ where $\Phi$ is solution of the dual linear equation and $A=-ia(x)(1-\Delta)^{-s}a(x)$, as in the linear control. The purpose is then to choose the adequat $\Phi_0$  and the system is completely determined.\\
With $\|r_0\|_{H^s}$ small enough, we are looking for a control such that $r(T)=0$. \\
More precisely, we consider the two systems
 \begin{eqnarray*}
\left\lbrace
\begin{array}{rcl}
i\partial_t \Phi + \Delta \Phi &=& 2|w|^2 \Phi - w^2 \overline{\Phi}\\
\Phi(x,0)&=&\Phi_{0}(x) \in H^{-s}
\end{array}
\right.
\end{eqnarray*}
and
\begin{eqnarray*}
\left\lbrace
\begin{array}{rcl}
i\partial_t r + \Delta r &=&2\left|w\right|^2r+w^2\bar{r}+F(w,r)+A\Phi\\
r(x,T)&=&0
\end{array}
\right.
\end{eqnarray*}

Let us define the operator
\begin{eqnarray*}
\begin{array}{rrcl}
L:&H^{-s}(M)&\rightarrow &H^s(M)\\
& \Phi_0&\mapsto &L \Phi_{0}=r(0).
\end{array}
\end{eqnarray*}

We split $r=v+\Psi$ with $\Psi$ solution of 
\begin{eqnarray*}
\left\lbrace
\begin{array}{rcl}
i\partial_t \Psi + \Delta \Psi &=&2|w|^2 \Psi + w^2 \overline{\Psi}+A\Phi\\
\Psi(T)&=&0
\end{array}
\right.
\end{eqnarray*}

This corresponds to the linear control, and so $\Psi(0)=S \Phi_0$. $v$ is solution of 
\begin{eqnarray}
\label{eqnv}
\left\lbrace
\begin{array}{rcl}
i\partial_t v + \Delta v &=&2|w|^2 v + w^2 \overline{v}+F(w,r)\\
v(T)&=&0
\end{array}
\right.
\end{eqnarray}
Then, $r$, $v$, $\Psi$ belong to $\Xsbt$ and $r(0)=v(0)+\Psi(0)$, which we can write
$$L\Phi_0= K\Phi_0+S\Phi_0$$
where $K\Phi_0=v(0)$. \\
$L\Phi_0 =r_0$ is equivalent to $\Phi_0=-S^{-1}K\Phi_0+S^{-1}r_0$. Defining the operator $B:H^{-s} \rightarrow H^{-s}$ by 
$$B\Phi_0=-S^{-1}K\Phi_0+S^{-1}r_0, $$
the problem $L\Phi_0 =r_0$ is now to find a fixed point of $B$ near the origin of $H^{-s}$. We will prove that $B$ is contracting on a small ball $B_{H^{-s}}(0,\eta)$ provided that $\nor{r_0}{H^s}$ is small enough.\\
We may assume $T<1$, and fix it for the rest of the proof (actually the norm of $S^{-1}$ depends on $T$ and even explode when $T$ tends to $0$, see \cite{Miller} and \cite{Tenenbaum}).\\
We have
$$\left\|B\Phi_0 \right\|_{H^{-s}}\leq C\left(\left\|K\Phi_0 \right\|_{H^{s}}+\left\|r_0 \right\|_{H^{s}}\right)$$
So, we are led to estimate $\left\|K\Phi_0 \right\|_{H^{s}}=\left\|v(0) \right\|_{H^{s}}$.\\
If we apply to equation (\ref{eqnv}) the estimate of Proposition \ref{thmlineareqn} we get
\bna
\left\|v(0) \right\|_{H^{s}}&\leq &\left\|v\right\|_{\Xsbt}\\
&\leq &C \left\|F(w,r)\right\|_{X^{s,-b'}_T}\\
&\leq & C \left\|w\right\|_{X^{1,b}_T}\left\|r\right\|^2_{X^{s,b}_T}+\left\|r\right\|^3_{X^{s,b}_T}
\ena
Then, we use the linear behavior near a trajectory of Proposition \ref{thmexistproche}. We conclude that for\\ $\left\|A\Phi\right\|_{L^2([0,T],H^s)}\leq \left\|\Phi\right\|_{X^{-s,b}_T} \leq C\left\|\Phi_0\right\|_{H^{-s}}<C\eta$ (see Proposition \ref{thmlineareqn}) small enough, we have 
\bna
\left\|r\right\|_{X^{s,b}_T}\leq C \left\|\Phi_0\right\|_{H^{-s}}.
\ena
This yields
\bna
\left\|B\Phi_0 \right\|_{H^{-s}}\leq C\left(\left\|\Phi_0\right\|^2_{H^{-s}}+\left\|\Phi_0\right\|^3_{H^{-s}}+\left\|r_0 \right\|_{H^{s}}\right)
\ena
Choosing $\eta$ small enough and $\left\|r_0 \right\|_{H^{s}}\leq \eta/2C$, we obtain $ \left\|B\Phi_0 \right\|_{H^{-s}}\leq \eta$ and $B$ reproduces the ball $B_{H^{-s}}(0,\eta)$.

If $u_0\in H^1$, we want one $g$ in $C([0,T],H^1)$, that is $\Phi_0\in H^{1-2s}$. We prove that $B$ reproduces $B_{H^{-s}}(0,\eta)\cap B_{H^{1-2s}}(0,R)$ for $R$ large enough.

Proposition \ref{thmpropagrecontrol} yields that $S$ is an isomorphism from $H^{1-2s}$ into $H^1$. Then, we have by the same arguments as above
$$\left\|B\Phi_0 \right\|_{H^{1-2s}}\leq C\left(\left\|K\Phi_0 \right\|_{H^{1}}+\left\|r_0 \right\|_{H^{1}}\right)$$
\bna
\left\|v(0) \right\|_{H^{1}}&\leq &C \left\|v\right\|_{\Xubt}\\
&\leq &C \left\|F(w,r)\right\|_{X^{1,-b'}_T}\\
&\leq & C \left\|w\right\|_{X^{1,b}_T}\left\|r\right\|_{X^{s,b}_T}\left\|r\right\|_{X^{1,b}_T}+\left\|r\right\|^2_{X^{s,b}_T}\left\|r\right\|_{X^{1,b}_T}
\ena
and 
$$\left\|r\right\|_{X^{1,b}_T}\leq C \left\|\Phi_0\right\|_{H^{1-2s}}.$$
Then,
$$\left\|B\Phi_0 \right\|_{H^{1-2s}}\leq C\left(R\eta+R\eta^2+ \left\|r_0 \right\|_{H^{1}}\right)$$
Choosing $\eta$ such that $C(\eta+\eta^2)<1/2$ (it is important to notice here that this bound does not depend on the size of $r_0$ in $H^1$) and $R$ large enough, we obtain that $B$ reproduces $B_{H^{-s}}(0,\eta)\cap B_{H^{1-2s}}(0,R)$.

Let us prove that it is contracting for the $H^{-s}$ norm. For that, we examine the systems
\begin{eqnarray}
\label{eqnNldiff}
\left\lbrace
\begin{array}{rcl}
i\partial_t (r-\tilde{r}) + \Delta (r-\tilde{r})&=& 2|w|^2(r-\tilde{r})+w^2\overline{(r-\tilde{r})}+F(w,r)-F(w,\tilde{r})+A(\Phi-\widetilde{\Phi})\\
(r-\tilde{r})(T)&=&0
\end{array}
\right.
\end{eqnarray}
\begin{eqnarray*}
\left\lbrace
\begin{array}{rcl}
i\partial_t (v-\tilde{v}) + \Delta (v-\tilde{v})  &=&2|w|^2(v-\tilde{v})+w^2\overline{(v-\tilde{v})}+F(w,r)-F(w,\tilde{r})\\
(v-\tilde{v})(T)&=&0
\end{array}
\right.
\end{eqnarray*}
We obtain
\bnan
\label{diffB}
\left\|B\Phi_0 -B\widetilde{\Phi}_0\right\|_{H^{-s}} & \leq & C \left\|(v-\tilde{v})(0)\right\|_{H^s} \leq  C\left\|F(w,r)-F(w,\tilde{r})\right\|_{X^{s,-b'}_T} \nonumber\\
&\leq & C \left(\left\|r\right\|_{\Xsbt}+\left\|\tilde{r}\right\|_{\Xsbt}+\left\|r\right\|^2_{\Xsbt}+\left\|\tilde{r}\right\|^2_{\Xsbt}\right)\left\|r-\tilde{r}\right\|_{\Xsbt} \nonumber \\
&\leq& C(\eta+\eta^2)\left\|r-\tilde{r}\right\|_{\Xsbt}\leq C\eta\left\|r-\tilde{r}\right\|_{\Xsbt}
\enan
Considering equation (\ref{eqnNldiff}), we deduce
\bna 
\left\|r-\tilde{r}\right\|_{\Xsbt} & \leq & C \left\|F(w,r)-F(w,\tilde{r})\right\|_{X^{s,-b'}_T} + C\left\|A(\Phi-\widetilde{\Phi})\right\|_{L^2([0,T],H^s)}\\
&\leq & C \eta \left\|r-\tilde{r}\right\|_{\Xsbt}+ C\left\| \Phi_0-\widetilde{\Phi}_0\right\|_{H^{-s}}
\ena
If $\eta$ is taken small enough it yields
\bnan
\label{estimdiff}
\left\|r-\tilde{r}\right\|_{\Xsbt}\leq C \left\| \Phi_0-\widetilde{\Phi}_0\right\|_{H^{-s}}.
\enan
Combining (\ref{estimdiff}) with (\ref{diffB}) we finally get
\bna
\left\|B\Phi_0 -B\widetilde{\Phi}_0\right\|_{H^{-s}} 
&\leq& C\eta \left\| \Phi_0-\widetilde{\Phi}_0\right\|_{H^{-s}}
\ena
This yields that $B$ is a contraction on a small ball $B_{H^{-s}}(0,\eta)$, which completes the proof of Proposition \ref{propcontrotraj}.
\end{proof}
\begin{corollaire}
Let $T>0$ and ($M,\omega)$ such that Assumptions \ref{geometriccontrol}, \ref{hypinegXsb}, \ref{uniquecontinuationH1} and \ref{hypcommut} are fulfilled.\\
Then, the set of reachable states is open in $H^s$ for $s_0< s \leq 1$.
\end{corollaire}
In the next corollary, $\widehat{f}(k)$ denotes the coordinates of a function $f$ in the basis of eigenfunction of $M$.
\begin{corollaire}
\label{corpetitharmon}
Suppose the same assumptions as Proposition \ref{propcontrotraj}.  Let $E_0>\nor{w_0}{H^1}$.\\
Then, there exist $N$ and $\varepsilon$ such that for every $u_0$ and $u_1\in H^1$ with
\bnan
\label{asumpetitharmon}
\left\|u_0\right\|_{H^1}\leq E_0 &\quad &\left\|u_1\right\|_{H^1} \leq E_0\\
\label{asumboundedH1}
\sum_{|k|\leq N} \left|\widehat{u_0}(k)-\widehat{w_0}(k)\right|^2 \leq \varepsilon &\quad&\sum_{|k|\leq N} \left|\widehat{u_1}(k)-\widehat{w_T}(k)\right|^2 \leq \varepsilon
\enan
we can find a control $g\in L^{\infty}([0,T],H^1)$ supported in $[0,T]\times \omega$ such that the unique solution of (\ref{eqnonlinsource}) with control $g$ and $u(0)=u_0$ satisfies $u(T)=u_1$.  
\end{corollaire}
\begin{proof}
We build the control in two steps : the first brings the system from $u_0$ to $w(T/2)$ and the second from $w(T/2)$ to $u_1$. Actually, the second step is the same by reversing time and we only describe the first one.

Let $s_0<s<1$. We first check that the first part of the conclusion of Proposition \ref{propcontrotraj} is true without Assumption \ref{hypcommut}. It gives one $\widetilde{\varepsilon}>0$ such that if $\left\|u_0-w_0\right\|_{H^s}\leq \widetilde{\varepsilon} $ we have a control to $w(T/2)$ in time $T/2$ with $g\in C([0,T/2],H^1)$. We only check that once $E_0$ is chosen, we can find $N$ and $\varepsilon$ such that assumptions (\ref{asumpetitharmon}) and (\ref{asumboundedH1}) imply $\left\|u_0-w_0\right\|_{H^s}\leq \widetilde{\varepsilon}$.
\end{proof}
We also obtain a first proof of global controllability. The Assumptions we make are stronger than Theorem \ref{thmcontrol} that will be proved using stabilization. Yet, in the examples we treat, the Assumptions are fulfilled.
\begin{corollaire}
\label{corglobtraj}
Theorem \ref{thmcontrol} is true under the stronger Assumptions \ref{geometriccontrol}, \ref{hypinegXsb} and \ref{uniquecontinuationH1}.
\end{corollaire}
\begin{proof}
We will make successives controls near some free nonlinear trajectory so that the energy decrease. The main argument is that the $\varepsilon$ of Theorem \ref{thmcontrolenonlin} only depends on $\left\|w\right\|_{\Xubt}$ and if the trajectory is a free nonlinear trajectory, then the $\varepsilon$ only depends on $\left\|w_0\right\|_{H^1}$.  We just have to be careful that each new free trajectory fulfills $\nor{w}{\Xubt}\leq A$ for one fixed constant $A$.
   
Fix $T>0$. There exist $C_1$ such that  
\bna
\left\|f\right\|_{H^1} \leq C_1\left(E(f) + \sqrt{E(f)}\right)^{1/2} \quad \forall f \in H^1(M).
\ena
Denote $A=C_1\left(E(w_0) + \sqrt{E(w_0)}\right)$. There exists a constant such that $\left\|w_0\right\|_{H^1}\leq A$ implies $\left\|w\right\|_{\Xubt}\leq B$ for $w$ solution of
\bna
\left\lbrace
\begin{array}{rcl}
i\partial_t w + \Delta w &=&|w|^2w \textnormal{ on } [0,T]\times M\\
w(0)&=&w_{0} 
\end{array}
\right.
\ena
Let $\varepsilon$ the constant so that Theorem \ref{thmcontrolenonlin} si true for any $w$ with $\left\|w\right\|_{\Xubt}\leq B$. We choose the arrival point $u_T=(1-\varepsilon/A)w_T$ such that  
$$\left\|u_T-w_T\right\|_{H^1}= \varepsilon/A \left\|w_T\right\|_{H^1} \leq C_1\left(E(w_T) + \sqrt{E(w_T)}\right) \varepsilon/A =  \varepsilon.$$ 
We have a control $g$ supported in $[0,T]\times \overline{\omega}$ such that the solution $u$ of
 \bna
\left\lbrace
\begin{array}{rcl}
i\partial_t u + \Delta u &=&|u|^2u +g\textnormal{ on } [0,T]\times M\\
u(0)&=&w_{0} 
\end{array}
\right.
\ena
satisfies $u(T)=u_T$.
If $ 1-\varepsilon/A\in [0,1]$, we have
\bna
E(u_T)=\frac{1}{2}\int_M |(1-\varepsilon/A) \nabla w_T|^2 +\frac{1}{4}\int_M |(1-\varepsilon/A) w_T|^4 \leq (1-\varepsilon/A)^2 E(w_T)
\ena
Moreover, we still have
\bna
\left\|u_T\right\|_{H^1} \leq C_1\left(E(u_T) + \sqrt{E(u_T)}\right)^{1/2}\leq A.
\ena
Then, we can reiterate this processus with the same $\varepsilon$. We construct a sequence of solutions $u_n\in X^{1,b}_{[nT,(n+1)T]}$ and of controls $g_n\in C([nT,(n+1)T],H^1)$ such that
 \bna
\left\lbrace
\begin{array}{rcl}
i\partial_t u_n + \Delta u_n &=&|u_n|^2u_n +g_n\textnormal{ on } [nT,(n+1)T]\times M\\
u_n(nT)&=&u_{n-1}(nT) 
\end{array}
\right.
\ena
and 
$$ E(u_n(nT))\leq (1-\varepsilon/A)^{2n}E(w_0)\leq C (1-\varepsilon/A)^{2n}\left(\left\|w_0\right\|^2_{H^1} +\left\|w_0\right\|^4_{H^1}\right)$$
But, we have
$$\left\|u_n(nT)\right\|^2_{H^1} \leq C_1\left(E(u_n(nT)) + \sqrt{E(u_n(nT))}\right)^{1/2}.$$
Therefore, it can be made arbitrary small for large $n$. This allows to use local controllability near the trajectory $0$. We obtain global controllability making the same proof in negative time.
\end{proof}
\section{Necessity of geometric control assumption on $S^3$}
\label{sectGCCnec}
In this section, we prove that on $S^3$, the geometric control is necessary for stabilization to occur. The argument uses some concentration of eigenfunctions. This concentration was also used by N. Burq, P. G\'erard and N. Tzvetkov \cite{instabiliteBGT} to prove some ill-posedness results.
\begin{prop}
Let $\Gamma$ be a geodesic of $S^3$ and $a\in C^{\infty}(S^3)$ such that $Supp(a)\cap \Gamma=\emptyset$. Then, for every $R_0>0$, $C$ and $\gamma>0$ there exist $T>0$ and $u_0 \in H^1(S^3)$ with $\nor{u_0}{H^1}\leq R_0$ such that
$$\nor{u(T)}{H^1} > Ce^{-\gamma T} \nor{u}{H^1}$$
for $u$ solution of equation 
\begin{eqnarray}
\label{eqndamped2}
\left\lbrace
\begin{array}{rcl}
i\partial_t u + \Delta u -(1+|u|^2)u&=&a(x)(1-\Delta)^{-1}a(x)\partial_t u \textnormal{ on } [0,T]\times S^3\\
u(0)&=&u_{0} \in H^1
\end{array}
\right.
\end{eqnarray}
\end{prop}
\begin{proof}Let $T$ such that $Ce^{-\gamma T}\leq 1/2$.

By changes of coordinates, we can assume that $\Gamma=\left\{x_3=x_4=0\right\}$. We will use the eigenfunctions $\Phi_n=c_n (x_1+ix_2)^n$ that concentrates on the subset $\left\{x_3=x_4=0\right\}$. $c_n$ is chosen such that $\nor{\Phi_n}{H^1}=R_0$ and so $c_n\approx n^{1/2-1}$. We have $-\Delta \Phi_n=\lambda_n \Phi_n$ with $\lambda_n= n(n+2)$. Let $u_n$ be the solution of (\ref{eqndamped2}) with $u_n(0)= \Phi_n$. Let $v_n=e^{i(\lambda_n-1)t}\Phi_n$ be the solution of the linear equation
\bna
\left\lbrace
\begin{array}{rcl}
i\partial_t v_n + \Delta v_n -v_n&=&0 \textnormal{ on } [0,T]\times S^3\\
v_n(0)&=& \Phi_n
\end{array}
\right.
\ena
Then, $r_n=u_n-v_n$ is solution of
\bna
\left\lbrace
\begin{array}{rcl}
i\partial_t r_n + \Delta r_n -r_n&=&a(x)(1-\Delta)^{-1}a(x)\partial_t r_n +R_n\textnormal{ on } [0,T]\times S^3\\
r_n(0)&=&0
\end{array}
\right.
\ena
with $R_n=|u_n|^2u_n+a(x)(1-\Delta)^{-1}a(x)\partial_t v_n$.

Proposition \ref{proplinearisation} of linearisation yields that $|u_n|^2u_n \longrightarrow 0$ in $X^{1,-b'}_T$.\\
For the other term in $R_n$, we use the concentration of the $\Phi_n$.
\bna
\nor{a(x)(1-\Delta)^{-1}a(x)\partial_t v_n}{X^{1,-b'}_T}&\leq &\nor{a(x)(1-\Delta)^{-1}a(x)\partial_t v_n}{L^2([0,T],H^1)}\\
&\leq &\nor{a(x)\partial_t v_n}{L^2([0,T],H^{-1})}\leq (\lambda_n+1)\nor{a(x)\Phi_n}{L^{\infty}(S^3)}
\ena
Let $\delta>0$, such that we have $x_3^2+x_4^2> \delta$ on Supp $a$. Hence, we have $\left|(x_1+ix_2)\right|^2=x_1^2+x_2^2=1-x_3^2-x_4^2<1-\delta$. 
\bna
(\lambda_n+1)\nor{a(x)\Phi_n}{L^{\infty}(S^3)}\leq C (\lambda_n+1) c_n  (1-\delta)^{n/2} 
\ena
Since $\lambda_n$ and $c_n$ are at most polynomial in $n$, we deduce that $R_n$ tends to $0$ in $X^{1,-b'}_T$. By some arguments similar to the proof of the continuity of the flow map of Proposition \ref{thmexistenceNl},, we infer that $r_n$ tends to $0$ in $X^{1,b}_T$. Then, $\nor{u_n(T)}{H^1}$ tends to $R_0$ and for $n$ large enough, we have $\nor{u_n(T)}{H^1}>R_0/2$. 
\end{proof}
With a similar proof, we could show the same result on $S^2\times S^1$ if $Supp(a)\cap (\Gamma \times S^1) =\emptyset$  for some geodesic $\Gamma$ of $S^2$. Yet, it does not imply geometric control.
 
The construction of J. V. Ralston \cite{Ralston} proves that actually, a necessary condition for stabilization is that the support of $a(x)$ intersects any stable closed geodesic (see also the work of L. Thomann \cite{Thomann} where this concentration is used to prove ill-posedness). In the case of $S^3$, we use the geometric fact that every closed geodesic is stable.      
\appendix
\section{Some commutator estimates}
\label{sectcommut}
This section is devoted to the proof of some commutator estimates used in Proposition \ref{thmpropagrecontrol}. More precisely, we study the action of $[(-\Delta)^{\varepsilon/2},a_1a_2]$ on $\Xsb$ where $a_i$ are rough. We first give a simple proof for $\Tot$ (rational or not) and then a general one under Assumption \ref{hypcommut}. Then, we show that this assumption is fulfilled for $S^3$ and $S^2\times S^1$. We will need an elementary lemma.

\begin{lemme}\label{lmepsilon} If $0\leq \varepsilon \leq 1$, we have for any norm $\left||k|^{\varepsilon}-|k_3|^{\varepsilon}\right| \leq \left|k-k_3\right|^{\varepsilon}$.
\begin{proof}
Using triangular inequality, we get $\left||k|-|k_3|\right|^{\varepsilon} \leq  \left|k-k_3\right|^{\varepsilon}.$ Then, we are reduced to the case of $\R^{+*}$ : we prove that for $z$, $t\in \R^{+*}$, we have $(z+t)^{\varepsilon}-z^{\varepsilon} \leq t^{\varepsilon}$, which is an easy consequence of Minkowsky inequality for $1\leq 1/\varepsilon \leq +\infty$. 
\end{proof}
\end{lemme}
\subsection{An easier proof for $\Tot$}
\begin{lemme}
Let $M=\R^3/(\theta_x \Z \times \theta_y \Z \times \theta_z\Z)$ with $(\theta_x ,\theta_y ,\theta_z) \in \R^3$. Denote $s_0$ the constant taken from Assumption \ref{hypinegXsb}. Let $s>s_0$ and $0\leq \varepsilon \leq 1$.\\
Then,  there exists $b'<1/2$ such that $u_3 \mapsto [\Delta^{\varepsilon/2},u_1u_2]u_3$ maps any $X^{s,b'}$ into $X^{s,-b'}$, where $u_1u_2$ denotes the operator of multiplication by $u_1u_2$ with $u_i\in X^{s+\varepsilon,b'}$ for $i\in \{1,2\}$.\\
This function $[\Delta^{\varepsilon/2},u_1u_2]$ also maps $X^{-s,b'}$ into $X^{-s,-b'}$.\\
Moreover, the same result holds with $u_i$ replaced by $\overline{u_i}$ for $i$ in a subset of $\{1,2,3\}$.
\end{lemme}
\begin{proof}We choose the norm $|k|=\sqrt{(\theta_xk_x)^2+(\theta_yk_y)^2+(\theta_zk_z)^2}$ so that 
$$\widehat{-\Delta u}(k)=|k|^2\widehat{u}(k) $$
By duality, it is equivalent to prove
$$\int_{\R\times M} [(-\Delta)^{\varepsilon/2},u_1u_2]u ~ \overline{v}\leq C \left\|u_1\right\|_{X^{s+\varepsilon,b'}} \left\|u_2\right\|_{X^{s+\varepsilon,b'}}\left\|u\right\|_{X^{s,b'}}\left\|v\right\|_{X^{-s,b'}}$$
Using Parseval theorem and denoting $k=k_1+k_2+k_3$, $\tau=\tau_1+\tau_2+\tau_3$
\bna
\int_{\R\times M} [(-\Delta)^{\varepsilon/2},u_1u_2]u ~ \overline{v} &=& \int_{\tau_1,\tau_2,\tau_3}\sum_{k_1,k_2,k_3} \widehat{u_1}(k_1,\tau_1)\widehat{u_2}(k_2,\tau_2)(|k|^{\varepsilon}-|k_3|^{\varepsilon})\widehat{u}(k_3,\tau_2) \overline{\widehat{v}}(k,\tau)\\
& \leq & \int_{\tau_1,\tau_2,\tau_3}\sum_{k_1,k_2,k_3} \left||k|^{\varepsilon}-|k_3|^{\varepsilon}\right| \left|\widehat{u_1}(k_1,\tau_1)\widehat{u_2}(k_2,\tau_2)\widehat{u}(k_3,\tau_3) \overline{\widehat{v}}(k,\tau)\right|
\ena
Lemma \ref{lmepsilon} and $k-k_3=k_1+k_2$ yields
\bna
\left|\int_{\R\times M} [(-\Delta)^{\varepsilon/2},u_1u_2]u ~ \overline{v} \right|
& \leq & C \int_{\tau_1,\tau_2,\tau_3}\sum_{k_1,k_2,k_3}\left(|k_1|^{\varepsilon}+|k_2|^{\varepsilon}\right) \left|\widehat{u_1}(k_1,\tau_1)\right|\left|\widehat{u_2}(k_2,\tau_2)\left|\right|\widehat{u}(k_3,\tau_3)\right| \left|\widehat{v}(k,\tau)\right|
\ena
Denoting $u_1^{§}$ the function with Fourier transform $  \left|\widehat{u_1}(k_1,\tau_1)\right|$ we obtain.
\bna
\left|\int_{\R\times M} [(-\Delta)^{\varepsilon/2},u_1u_2]u ~ \overline{v} \right|
& \leq & C \int_{\R\times M} \left(\Delta^{\varepsilon/2}u_1^{§}\right)u_2^{§} u^{§} ~ \overline{v^{§}} +\int_{\R\times M} u_1^{§}\left(\Delta^{\varepsilon/2}u_2^{§}\right) u^{§} ~ \overline{v^{§}} \\
&\leq & C \left\|u_1\right\|_{X^{s+\varepsilon,b'}}\left\|u_2\right\|_{X^{s+\varepsilon,b'}}\left\|u\right\|_{X^{s,b'}} \left\|v\right\|_{X^{-s,b'}}
\ena
Here, we have finished the proof using the trilinear Bourgain estimate because $s>s_0$.
If we estimate this integral using the trilinear estimate at the negative level $H^{-s}$, we obtain the second result we announced.
\end{proof}
\subsection{General proof under Assumption \ref{hypcommut}}
\begin{lemme}
\label{quadrdoneXsb}
Denote $s_0$ the constant taken from Assumption \ref{hypcommut}. Let $s>s_0$ and $0\leq \varepsilon \leq 1$.\\
Then,  there exists $b'<1/2$ such that $u_3 \mapsto [(-\Delta)^{\varepsilon/2},u_1u_2]u_3$ maps any $X^{s,b'}$ into $X^{s,-b'}$, where $u_1ua_2$ denotes the operator of multiplication by $u_1u_2$ with $u_i\in X^{s+\varepsilon,b'}$ for $i\in \{1,2\}$.\\
This function $[\Delta^{\varepsilon/2},u_1u_2]$ also maps $X^{-s,b'}$ into $X^{-s,-b'}$.\\
Moreover, the same result holds with $u_i$ replaced by $\overline{u_i}$ for $i$ in a subset of $\{1,2,3\}$.
\end{lemme}
\begin{proof}
The proof follows the techniques of J. Bourgain and N. Burq, P. Gérard, N. Tzvetkov. Here, we were inspired more precisely by \cite{PGPierfelice}. We recall the notations $u^{\#}=e^{-it\Delta}u(t)$, $u^{N}=\mathbf{1}_{\sqrt{1-\Delta}\in [N,2N[}u$ where $N$ is a dyadic number and $\widehat{u}(\tau)$ is the Fourier transform of $u$ with respect to the time variable. First, with some dyadic integers $N_i$ fixed, we estimate the integral 
\bna
&&I(N_1,..,N_4)=\int_{\R\times M}u_1^{N_1}u_2^{N_2}\left[((-\Delta)^{\varepsilon/2}u_3^{N_3}) ~ \overline{u_4}^{N}-u_3^{N_3} ~ (-\Delta)^{\varepsilon/2}\overline{u_4}^{N}\right]~dtdx\\
&=&\frac{1}{(2\pi)^4}\int_{\R\times M}\iiiint_{\R^4}e^{it(\tau_1+\tau_2+\tau_3-\tau)}e^{it\Delta}\widehat{u_1^{N_1\#}}(\tau_1)e^{it\Delta}\widehat{u_{2}^{N_2\#}}(\tau_2)\\
&&\times\left[((-\Delta)^{\varepsilon/2}e^{it\Delta}\widehat{u_3^{N_3\#}}(\tau_3))\overline{e^{it\Delta}\widehat{u_4^{N\#}}(\tau)}-e^{it\Delta}\widehat{u_3^{N_3\#}}(\tau_3)(-\Delta)^{\varepsilon/2}\overline{e^{it\Delta}\widehat{u_4^{N_4\#}}(\tau)}\right]d\tau_1~d\tau_2~d\tau_3d\tau_4~dtdx
\ena
By nearly orthogonality in $H^b$ and partition of unity, $u_j=\sum_{n\in \Z} \varphi(t-n/2)u_j(t)$, we are led to the special case where the $u_j$ are supported in time in the interval $]0,1[$. Select $\chi \in C^{\infty}_0(\R)$ such that $\chi=1$ on $[0,1]$. Thus, estimates (\ref{quadrinlinschrod}), applied with $\tau_j$ fixed, and Cauchy-Schwarz inequality in $(\tau_1,\tau_2,\tau_3,\tau_4)$ gives for any $b>1/2$
\bnan
\label{estimatemultinopt}
\left|I(N_1,..,N_4)\right|&\leq &C (N_1^{\varepsilon}+N_2^{\varepsilon})\left(m(N_1,\troisp,N_4)\right)^{s_0}\prod_{j=1}^4\int_{\tau_j}\nor{\widehat{u_j^{N_j\#}}(\tau_j)}{L^2(M)}\nonumber\\
&\leq& C (N_1^{\varepsilon}+N_2^{\varepsilon})\left(m(N_1,\troisp,N_4)\right)^{s_0}\prod_{j=1}^4\nor{u_j^{N_j}}{X^{0,b}(\R\times M)}
\enan
This estimate is very satisfactory for the space regularity. Yet, for the regularity in time, it requires $b>1/2$ which is too much for our purpose. We will interpolate with some crude estimates in space but better in time.\\
For the case where $N_1$ is large, we estimate $\left|I(N_1,\troisp,N)\right|$ using Sobolev embeddings $H^{1/4}(\R)\subset L^4(\R)$  :
\bnan
\label{estimatemultilinsobo}
\left|I(N_1,\troisp,N_4)\right|\leq C (N_3^{\varepsilon}+N_4^{\varepsilon}) \left(m(N_1,\troisp,N_4)\right)^{3/2} \prod_{j=1}^4\nor{u_j^{N_j}}{X^{0,1/4}(\R\times M)}
\enan
In another case where the frequency $N_3$ is large, we will use an argument near \cite{DelortSJ}. In that case, we can not afford a loss in the frequency $N_3$. We use the fact that $[u_1^{N_1}u_2^{N_2},\Delta^{\varepsilon/2}]$ is a pseudodifferential operator of order less than $0$ (if $\varepsilon \leq 1$). 
Then,
\bnan
\left|I(N_1,\troisp,N_4)\right|&=&\left|\int_{\R\times M}[u_1^{N_1}u_2^{N_2},\Delta^{\varepsilon/2}]u_3^{N_3}\overline{u_4^{N_4}}\right| \nonumber\\
&\leq & C\int_{\R} \left\|[u_1(t)^{N_1}u_2(t)^{N_2},\Delta^{\varepsilon/2}]\right\|_{L^2\rightarrow L^2 } \left\|u_3(t)\right\|_{L^2(M)}\left\|u_4(t)\right\|_{L^2(M)}dt\nonumber\\
&\leq& \int_{\R}\sum_{\alpha=0}^m\nor{\partial^{\alpha} u_1u_2(t)}{L^{\infty}(M)}\left\|u_3(t)\right\|_{L^2(M)}\left\|u_4(t)\right\|_{L^2(M)}dt\nonumber\\
\label{estimcommutbis} &\leq& C \max\left(N_1, N_2\right)^{\mu} \prod_{j=1}^4\nor{u_j^{N_j}}{X^{0,1/4}(\R\times M)} 
\enan
where $\mu$ depends on the dimension and on $\varepsilon$.

Let us now begin the summation of the harmonics. As in \cite{PGPierfelice}, we decompose each function
$$u=\sum_{K}u_{K}, \quad u_{K}=\mathbf{1}_{K\leq \left\langle i\partial_t+\Delta \right\rangle<2K}(u) $$
where $K$ denotes the sequence of dyadic integers. Notice that
$$\left\|u\right\|_{X^{0,b}}^2 \approx \sum_{K}K^{2b}\nor{u_K}{L^2(\R\times M)}^2\approx \sum_{K}\nor{u_K}{X^{0,b}}^2.$$
Then, we decompose the integral in sum of the following elementary integrals
\bna
I(N_1,\troisp,N_4,K_1,\troisp, K_4)=\int_{\R\times M}a_1^{N_1,K_1}a_2^{N_2,K_2}\left[((-\Delta^{\varepsilon/2})u^{N_3,K_3})\overline{v}^{N,K}-u^{N_3,K_3} (-\Delta^{\varepsilon/2})\overline{v}^{N,K}\right]dtdx
\ena
Estimate (\ref{estimatemultinopt}) leads to (for every $b>1/2$)
\bna
\left|I(N_1,\troisp,N_4,K_1,\troisp,K_4)\right| \leq  (N_1^{\varepsilon}+N_2^{\varepsilon})m(N_1,\troisp,N_4)^{s_0}(K_1K_2K_3K_4)^b\prod_{j=1}^4 \norL{u_j^{N_j,K_j}}.
\ena
We will interpolate this estimate with different inequalities. We distinguish three cases : $N_4\leq C(N_1+N_2+N_3)$ with $N_3<\max(N_1,N_2)$ or $\max(N_1,N_2)\leq N_3$, and the last case $N_4> C(N_1+N_2+N_3)$ with $C$ large enough. Without loss of generality, we can assume $N_1\geq N_2$.

First case : $N_3<\max(N_1,N_2)=N_1$ and $N_4\leq C(N_1+N_2+N_3)$\\
Estimate (\ref{estimatemultilinsobo}) gives 
\bna
\left|I(N_1,\troisp,N_4,K_1,\troisp,K_4)\right|& \leq & (N_3^{\varepsilon}+N_4^{\varepsilon})m(N_1,\troisp,N_4)^{3/2}\\
&&(K_1K_2K_3K_4)^{1/4} \prod_{j=1}^4 \norL{u_j^{N_j,K_j}}
\ena
Then, for every $\theta \in [0,1]$
\bna
\left|I(N_1,\troisp,N_4,K_1,\troisp,K_4)\right| & \leq & C (N_1^{\varepsilon}+N_2^{\varepsilon})^{1-\theta} (N_3^{\varepsilon}+N_4^{\varepsilon})^{\theta}m(N_1,\troisp,N_4)^{(1-\theta)s_0+3\theta/2}\\
&&(K_1K_2K_3K)^{b(1-\theta)+\theta/4} \prod_{j=1}^4 \norL{u_j^{N_j,K_j}}
\ena
We denote $s(\theta)=(1-\theta)s_0+3\theta/2$ and $b(\theta)=b(1-\theta)+\theta/4$. 
\bna
\left|I(N_1,\troisp,N,K_1,\troisp)\right|& \leq &  C(N_1^{\varepsilon}+N_2^{\varepsilon})^{1-\theta} (N_3^{\varepsilon}+N_4^{\varepsilon})^{\theta}m(N_1,\troisp,N_4)^{s(\theta)}\\
&&(K_1K_2K_3K_4)^{b(\theta)-b'} \prod_{j=1}^4 \nor{u_j^{N_j}}{X^{0,b'}}
\ena
By choosing some appropriate $\theta$ and $b'<1/2<b$, we can make the serie in $K$ convergent if $b(\theta)-b'<0$. This yields :
\bna
&&\left|I(N_1,\troisp,N_4)\right| \leq C(N_1^{\varepsilon}+N_2^{\varepsilon})^{1-\theta} (N_3^{\varepsilon}+N_4^{\varepsilon})^{\theta}m(N_1,\troisp,N_4)^{s(\theta)}\prod_{j=1}^4 \nor{u_j^{N_j}}{X^{0,b'}}\\
&&\leq C N_1^{(1-\theta)\varepsilon-s-\varepsilon}N_4^{s+\theta\varepsilon}N_2^{s(\theta)-s-\varepsilon}N_3^{s(\theta)+\theta\varepsilon-s} \prod_{j=1}^2\nor{u_j}{X^{s+\varepsilon,b'}}\nor{u_3}{X^{s,b'}}\nor{u_4}{X^{-s,b'}}\\
&&\leq \left(\frac{N_4}{N_1}\right)^{s+\theta\varepsilon}N_2^{s(\theta)-s-\varepsilon}N_3^{s(\theta)+\theta\varepsilon-s} \prod_{j=1}^2\nor{u_j}{X^{s+\varepsilon,b'}}\nor{u_3}{X^{s,b'}}\nor{u_4}{X^{-s,b'}}\\
\ena
The series is convergent thanks to $N_4\leq CN_1$ and after choosing $\theta$ small enough such that $s(\theta)+\theta\varepsilon-s<0$ with $b(\theta)-b'<0$.

Second case : $N_1=\max(N_1,N_2)\leq N_3$ and so $N_4\leq CN_3$.\\
This time, $N_3$ is a large frequency and we can not have any loss $N_3^{\theta\varepsilon}$ from the interpolation. We proceed with the same interpolation procedure but between (\ref{estimatemultinopt}) and (\ref{estimcommutbis}). After summation in $K$ and a good choice of $b'<1/2<b$ and
\bna
\left|I(N_1,\troisp,N_4)\right|&\leq &C N_1^{(1-\theta)(s_0+\varepsilon)+\theta\mu-s-\varepsilon}N_4^{s}N_2^{(1-\theta)s_0-s-\varepsilon}N_3^{-s} \prod_{j=1}^2\nor{u_j}{X^{s+\varepsilon,b'}}\nor{u_3}{X^{s,b'}}\nor{u_4}{X^{-s,b'}}\\
&\leq & \left(\frac{N_4}{N_3}\right)^{s}N_1^{(1-\theta)(s_0+\varepsilon)+\theta\mu-s-\varepsilon}N_2^{(1-\theta)s_0-s-\varepsilon} \prod_{j=1}^2\nor{u_j}{X^{s+\varepsilon,b'}}\nor{u_3}{X^{s,b'}}\nor{u_4}{X^{-s,b'}}\\
\ena
We choose $\theta$ small enough such that $(1-\theta)(s_0+\varepsilon)+\theta\mu-s-\varepsilon \leq s_0+\theta\mu-s<0$ and $b(\theta)-b'<0$. And we conclude by the same summation as in the first case.

Last case : $N_4 \geq C(N_1+N_2+N_3)$\\
This case is trivial in the particular case of $\Tot$, $S^3$ or $S^2\times S^1$ since this integral is zero for $C$ large enough. In the general case, we apply the following lemma  which is a variant of Lemma 2.6 in \cite{InventionesBGT}.
\begin{lemme}
There exists $C>0$ such that, if for any $j=1,2,3$, $C\mu_{k_j}\leq \mu_{k_4}$, then for every $p>0$, there exists $C_p>0$ such that for every $w_j\in L^2(M)$, $j=1,2,3,4$
\bna
\int_{ M}\Pi_{k_1}w_1\Pi_{k_2}w_2\left[(-\Delta)^{\varepsilon/2}\Pi_{k_3}w_3 ~ \overline{\Pi_{k_4}w_4}-\Pi_{k_3}w_3 ~ (-\Delta)^{\varepsilon/2}\overline{\Pi_{k_4}w_4}\right]dx \leq C_p \mu_{k_4}^{-p} \prod_{j=1}^4 \norL{w_j}
\ena
where $\Pi_{k}$ denotes the orthogonal projection on the eigenfunction $e_k$ associated to the eigenvalue $\mu_k$.
\end{lemme}
This ends the proof of the fist statement of Lemma \ref{quadrdoneXsb}. The second one is obtained by duality.
\end{proof}
\subsection{$S^3$ and $S^2\times S^1$ fulfill Assumption \ref{hypcommut}}
\begin{lemme}
Assumption \ref{hypcommut} holds true with any $s_0>1/2$ on $S^3$ and any $s_0>3/4$ on $S^2\times S^1$. 
\end{lemme}
\begin{proof}
We first treat the case of $S^3$ and follow the scheme of Proposition 3 of \cite{PGPierfelice}.
We write 
\bna
f_j=\sum_{n_j}H^{(j)}_{n_j},
\ena
where $H^{(j)}_{n_j}$ are spherical harmonics of degree $n_j$, and where the sum on $n_j$ bears on the domain
\bnan \label{spectredyadN} N_j\leq \sqrt{1+n_j(n_j+2)}< 2N_j.  \enan
Then, the solution $u_j$ are given by
$$u_j(t)=e^{it\Delta}f_j=\sum_{n_j}e^{-itn_j(n_j+2)}H^{(j)}_{n_j} $$
and we have to estimate
\bna
Q(f_1,\troisp,f_4,\tau)&=&\int_{\R}\int_{S^3} \chi(t)e^{it\tau}u_1u_2\left[(-\Delta)^{\varepsilon/2}u_3\overline{u_4}-u_3(-\Delta)^{\varepsilon/2}\overline{u_4}\right]dxdt\\
&= &\sum_{n_1,\troisp,n_4}\widehat{\chi}(\sum_{j=1}^4 \varepsilon_j n_j(n_j+2)-\tau)I(H_{n_1}^{(1)},\troisp,H_{n_4}^{(4)}),
\ena
with $\varepsilon_j=-1$ or $1$ depending on the position of conjugates and
\bna
I(H_{n_1}^{(1)},\troisp,H_{n_4}^{(4)})=(\sqrt{n_3(n_3+2)}^{\varepsilon}-\sqrt{n_4(n_4+2)}^{\varepsilon})\int_{S^3}H_{n_1}^{(1)}H_{n_2}^{(2)}H_{n_3}^{(3)}\overline{H}_{n_4}^{(4)}dx
\ena
We notice that $\int H_{n_1}H_{n_2}H_{n_3}\overline{H}_{n_4}\neq 0$ implies $n_4\leq n_1+n_2+n_3$ and $n_3\leq n_1+n_2+n_4$, that is $|n_4-n_3|\leq n_1+n_2$. Then, using Lemma \ref{lmepsilon} and fundamental theorem of calculus, we have
\bnan
\left|\sqrt{n_3(n_3+2)}^{\varepsilon}-\sqrt{n_4(n_4+2)}^{\varepsilon}\right|&\leq &\left|\sqrt{n_3(n_3+2)}- \sqrt{n_4(n_4+2)}\right|^{\varepsilon} \nonumber\\
\label{inegcommutn} &\leq &C \left| n_4-n_3\right|^{\varepsilon}\leq C (N_1^{\varepsilon}+N_2^{\varepsilon})
\enan
Moreover, bilinear eigenfunctions estimates (see Theorem 2 of \cite{Xsbsphere} or Theorem 2.5 of \cite{gerardcourspise}) yield 
\bna
\left|I(H_{n_1}^{(1)},\troisp,H_{n_4}^{(4)})\right|&\leq& C (N_1^{\varepsilon}+N_2^{\varepsilon}) \left|\int_{S^3}H_{n_1}^{(1)}H_{n_2}^{(2)}H_{n_3}^{(3)}\overline{H}_{n_4}^{(4)}dx\right|\\
&\leq &C (N_1^{\varepsilon}+N_2^{\varepsilon})m(N_1,\troisp,N_4)^{1/2+}\prod_{j=1}^4 \norL{H_{n_j}^{(j)}}
\ena
Using the fast decay of $\widehat{\chi}$ at infinity, we infer
\bna
\left|Q(f_1,\troisp,f_4,\tau)\right|&\leq & C (N_1^{\varepsilon}+N_2^{\varepsilon}) m(N_1,\troisp,N_4)^{1/2+}\sum_{l\in \Z} (1+\left|l\right|^2)^{-1}\sum_{\Lambda([\tau]+l)}\prod_{j=1}^4 \norL{H_{n_j}^{(j)}}\\
&\leq &C (N_1^{\varepsilon}+N_2^{\varepsilon}) m(N_1,\troisp,N_4)^{1/2+}\sup_{k\in \Z} \sum_{\Lambda(k)}\prod_{j=1}^4 \norL{H_{n_j}^{(j)}}
\ena
where $\Lambda(k)$ denotes the set of $(n_1,\troisp,n_4)$ satisfying (\ref{spectredyadN}) for $j=1,2,3,4$ and
$$\sum_{j=1}^4 \varepsilon_j n_j(n_j+2)=k.$$
Now, we write
$$\left\{1,2,3,4\right\}=\left\{\alpha,\beta,\gamma,\delta\right\} $$
with $m(N_1,\troisp,N_4)=N_{\alpha}N_{\beta}$ and we split the sum on $\Lambda(k)$ as
\bna
\left|Q(f_1,\troisp,f_4,\tau)\right|&\leq &C (N_1^{\varepsilon}+N_2^{\varepsilon}) m(N_1,\troisp,N_4)^{1/2+}\sup_{k\in \Z} \sum_{a\in \Z} S(a)S'(k-a)
\ena
where
\bna
S(a)=\sum_{\Gamma(a)}\norL{H_{n_{\alpha}}^{(\alpha)}}\norL{H_{n{\gamma}}^{(\gamma)}};\quad S'(a')=\sum_{\Gamma'(a')}\norL{H_{n_{\beta}}^{(\beta)}}\norL{H_{n_{\delta}}^{(\delta)}},\\
\Gamma(a)=\{(n_{\alpha},n_{\gamma}) : (\ref{spectredyadN}) \textnormal{ holds for } j=\alpha, \gamma, \sum_{j=\alpha,\gamma} \varepsilon_j n_j(n_j+2)=a\},\\
\Gamma'(a')=\{(n_{\beta},n_{\delta}) : (\ref{spectredyadN}) \textnormal{ holds for } j=\beta, \delta, \sum_{j=\beta,\delta} \varepsilon_j n_j(n_j+2)=a'\}.
\ena
Then, we use a number theoretic result involving the ring of Gauss integers (see Lemma 3.2 of \cite{InventionesBGT}). 
\begin{lemme}
Let $\sigma\in \{\pm 1 \}$. For every $\eta>0$, there exists $C_{\eta}$ such that, given $M\in \Z$ and a positive integer $N$,
$$\# \{(k_1,k_2) \in \N^2 : N\leq k_1\leq 2N, k_1^2+\sigma k_2^2=M\} \leq C_{\eta} N^{\eta}.$$
\end{lemme}
Noticing that $n_j(n_j+2)=(n_j+1)^2-1$, we get  
$$\sup_{a}\#\Gamma(a)\leq C_{\eta}N_{\alpha}^{\eta};\quad \sup_{a'}\# \Gamma'(a')\leq C_{\eta}N_{\beta}^{\eta},$$
and consequently, by the Cauchy-Schwarz inequality and the orthogonality of the $H_{n_{j}}^{(j)}$
\bna
&&\sum_{a\in \Z} S(a)S'(k-a) \leq C_{\eta}(N_{\alpha}N_{\beta})^{\eta/2}\times\\
&&\left(\sum_{a}\sum_{\Gamma(a)} \norL{H_{n_{\alpha}}^{(\alpha)}}^2\norL{H_{n_{\gamma}}^{(\gamma)}}^2 \right)^{1/2}\left(\sum_{a}\sum_{\Gamma'(k-a)} \norL{H_{n_{\beta}}^{(\beta)}}^2 \norL{H_{n_{\delta}}^{(\delta)}}^2\right)^{1/2}\\
&&\leq C_{\eta}(N_{\alpha}N_{\beta})^{\eta/2} \prod_{j=1}^4 \norL{f_j}.
\ena
This completes the proof for $S^3$.

 For $S^2\times S^1$, we adapt this argument with some slight modifications.\\
First, the formulae should be changed to 
$$u_j(t)(x,y)=e^{it\Delta}f_j=\sum_{n_j,p_j}e^{-itn_j(n_j+1)-ip_j^2t}H^{(j)}_{n_j,p_j}(x)e^{ip_jy} $$
where $H^{(j)}_{n_j,p_j}$ are spherical harmonics on $S^2$ of degree $n_j$. Estimate (\ref{inegcommutn}) becomes
\bna
\left|\sqrt{n_3(n_3+1)+p_3^2}^{\varepsilon}-\sqrt{n_4(n_4+2)+p_4^2}^{\varepsilon}\right|&\leq& \left|\sqrt{n_3(n_3+1)+p_3^2}-\sqrt{n_4(n_4+1)+p_4^2}\right|^{\varepsilon}\\
&\leq & \left|\left[\sqrt{n_3(n_3+1)}-\sqrt{n_4(n_4+1)}\right]^2+(p_3-p_4)^2\right|^{\varepsilon/2}\\
&\leq& \left|C(n_3-n_4)^2+(p_3-p_4)^2\right|^{\varepsilon/2}\\
&\leq & C \left|(n_1+n_2)^2+(p_1+p_2)^2\right|^{\varepsilon/2}\leq C (N_1^{\varepsilon}+N_2^{\varepsilon})
\ena
where we have used $|n_3-n_4|\leq |n_1+n_2|$ and $|p_3-p_4|\leq |p_1|+|p_2|$ for the integral to be non zero.
Bilinear eigenfunctions estimates for $S^2$ yield 
\bna
\left|I(H_{n_1,p_1}^{(1)},\troisp,H_{n_4,p_4}^{(4)})\right|&\leq &C (N_1^{\varepsilon}+N_2^{\varepsilon})m(N_1,\troisp,N_4)^{1/4}\prod_{j=1}^4 \norL{H_{n_j,p_j}^{(j)}}.
\ena
We finish the proof similarly, replacing the formula for $\Gamma(a)$ by
\bna
\Gamma(a)=&&\{(n_{\alpha},p_{\alpha},n_{\gamma},p_{\gamma}) : N_j\leq \sqrt{1+n_j(n_j+2)+p_j^2}\leq 2N_j,j=\alpha,\gamma \\ 
&&\textnormal{ and } \sum_{j=\alpha,\gamma} \varepsilon_j [n_j(n_j+2)+p_j^2]=a\}
\ena
In that case, the same number theoretic arguments yield $\sup_a \# \Gamma(a)\leq C_{\eta} N^{1+\eta}_{\alpha} $ and finally, after Cauchy-Schwarz inequality, we obtain
 $$\left|Q(f_1,\troisp,f_4,\tau)\right| \leq C (N_1^{\varepsilon}+N_2^{\varepsilon})m(N_1,\troisp,N_4)^{1/4+(1+\eta)/2} \prod_{j=1}^4 \norL{f_j}.$$
\end{proof}
\section{Unique continuation}
\label{sectuniqueness}
\subsection{Carleman estimates}
This section is only a variant in the Riemannian setting of some results of A. Mercado, A. Osses and L. Rosier \cite{CarlemanSchrodRosier}. We follow their proof very closely, sometimes line by line.\\
For sake of simplicity, we will assume that $u$ is supported in a fixed compact $K$ of a Riemannian manifold $\Omega$. Yet, the same reasonning as in \cite{CarlemanSchrodRosier} would allow to handle the case of Dirichlet boundary conditions for $u$. We have changed the notation of the manifold from $M$ to $\Omega$ because the Carleman estimates will not be used on the whole compact manifold $M$ but only on some open set $\Omega$.\\
$D$ denotes the Levi-Civita connection associated to the metric $g$. Then, it is torsion-free and the Hessian of the functions are symmetrics.\\
 $\cdot\quad$, $|\quad|$, $\nabla$ and $\Delta$ denote the scalar product, the norm, the gradient and the Laplacian with respect to the metric $g$. Moreover, the scalar product will be the real one : if $X=a+ib$ and $Y=c+id$, $X.Y=a\cdot c-b\cdot d +i(b\cdot c +a\cdot d)$ and $|X|^2=X\cdot \overline{X}$. $v_g$ denotes the Riemannian volume form and all the integrals are defined with this (even if it will be often omitted).\\
First, we list a few formulae that will be used along the proof. For any functions $f, h\in C^{\infty}(\Omega)$ with $h$ compactly supported and any vector fields $X$, $Y$ and $Z$, we have
\bna
D_{Z} (X\cdot Y)&=&  (D_{Z}X)\cdot Y +X\cdot (D_{Z}Y)\\
\nabla f \cdot Z &=& D_{Z} f \\
(D_{X}\nabla f )\cdot Y &=& Hess(f)(X,Y)\\
\int_{\Omega} (\Delta f) h ~dv_g&=& - \int_{\Omega} \nabla f \cdot \nabla h ~dv_g\\
\nabla (fh)&=& (\nabla f)h+f(\nabla h)\\
div(fX)&=&fdiv(X)+X\cdot \nabla f 
\ena

\bigskip
 
For brevity, $\iint$ will denote the integral over $]-T,T[\times \Omega$ and $\iint_{\omega}$ the integral over $]-T,T[\times \omega$ where $\omega$ is an open subset of $\Omega$. 

Let $\Psi\in C^4(\Omega)$ real valued . We assume that $\Psi$ satisfies the following properties
\bnan
\nabla\Psi\neq0 \textnormal{ in } \Omega  \backslash~ \omega \label{gradnonnul}\\
\Psi(x)\geq 2/3 \left\|\Psi\right\|_{L^{\infty}}. \label{assumCpsi}
\enan
(\ref{assumCpsi}) is technical and is easily fulfilled by replacing $\Psi$ by $\Psi+C$ with $C$ large enough.
We distinguish two cases : strong pseudoconvexity and weak pseudoconvexity.\\
The case of strong pseudoconvexity can be found in Isakov\cite{Isakov} but with local in time estimates, it reads
\bnan
Hess(\Psi(x))(\xi,\xi)+\left|\nabla \Psi(x) \cdot \xi\right|^2>0 \quad \forall (x,\xi) \in  T\Omega \backslash~ T\omega, 
\enan
which implies since the support is compact that 
\bnan
\label{strongconvex}
Hess(\Psi(x))(\xi,\xi)+\left|\nabla \Psi(x) \cdot \xi\right|^2>C \left|\xi \right|^2 \quad \forall (x,\xi) \in  T\Omega \backslash~ T\omega,~x\in K
\enan
Weak pseudoconvexity is defined by
\bnan
\label{weakconvex}
Hess(\Psi(x))(\xi,\xi)+\left|\nabla \Psi(x) \cdot \xi\right|^2\geq0 \quad \forall (x,\xi) \in  T\Omega \backslash~ T\omega.
\enan
Set $C_{\Psi}=2\left\|\Psi \right\|_{L^{\infty}(\Omega)}$ and
\bna
\theta(t,x):=\frac{e^{\lambda\Psi(x)}}{(T-t)(T+t)}, \quad \varphi(t,x):=\frac{e^{\lambda C_{\Psi}}-e^{\lambda \Psi(x)}}{(T-t)(T+t)},\quad \forall (t,x) \in ]-T,T[\times \Omega
\ena
Denote by $L(q)=i\partial_tq+\Delta q$ the linear Schrödinger operator.
\begin{prop}
\label{propCarlemanstrong}
Let $T>0$. Let $\Omega$ be a Riemannian manifold and $K$ a compact subset of $\Omega$.    
Assume that there exists a function $\Psi\in C^4(\Omega)$ such that (\ref{gradnonnul}), (\ref{assumCpsi}) and (\ref{strongconvex}) hold for some open set $\omega\subset \Omega$. Then, there exist constants $\lambda_0$, $s_0$ and $C$ such that for all $\lambda\geq \lambda_0$, all $s\geq s_0$ and $q\in L^2(]-T,T[,H^1(\Omega))$, supported in $K$, with $L(q)\in L^2(]-T,T[\times \Omega)$ we have
\bnan
\label{Carlemanstrong}
&&\iint \left[s^3\lambda^4 \theta^3 |q|^2+s\lambda  \theta \left|\nabla q\right|^2\right]e^{-2s\varphi}\\
&\leq & C\iint\left|L(q)\right|^2e^{-2s\varphi}+C\iint_{\omega} \left[s^3\lambda^4 \theta^3|q|^2+s\lambda\theta \left|\nabla q \right|^2\right]e^{-2s\varphi}\nonumber
\enan
\end{prop}
\begin{prop}
\label{propCarlemanweak}
If in Proposition \ref{propCarlemanstrong}, we replace Assumption (\ref{strongconvex}) by (\ref{weakconvex}), we obtain the same result with
\bnan
\label{Carlemanweak}
&&\iint \left[s^3\lambda^4 \theta^3 |q|^2+s\lambda^2  \theta \left|\nabla \Psi\cdot\nabla q \right|^2\right]e^{-2s\varphi}\\
&\leq & C\iint\left|L(q)\right|^2e^{-2s\varphi}+C\iint_{\omega} \left[s^3\lambda^4 \theta^3|q|^2+s\lambda\theta \left|\nabla q \right|^2\right]e^{-2s\varphi}\nonumber
\enan
\end{prop}
\begin{proof}
Using regularisation in a standard way, we are reduced to consider $q\in C^{\infty}(]-T,T[\times \Omega)$. Denote $u=e^{-s\varphi}q$ and $w=e^{-s\varphi}L(q)=e^{-s\varphi}L(e^{s\varphi}u)$. We notice that $u$ and all its time derivatives vanish at $t=-T$ and $t=T$. Thus, all the integrations by part in time do not create any boundary term. We compute
\bna
w=Pu=iu_t+is\varphi_tu+\Delta u+2s\nabla \varphi \cdot \nabla u+s(\Delta \varphi) u+s^2|\nabla \varphi|^2u
\ena
We decompose $P=P_1+P_2$ with
\bna
&&P_1u:=is\varphi_tu+2s\nabla \varphi \cdot \nabla u+s(\Delta \varphi) u\\
&&P_2u:=iu_t+\Delta u+s^2|\nabla \varphi|^2u
\ena
\bna
\left\|w\right\|^2_{L^2(-T,T[\times \Omega)}=\left\|P_1u+P_2u\right\|^2=\left\|P_1u\right\|^2+\left\|P_2u\right\|^2+2\Re (P_1u,P_2u)
\ena
As usual in Carleman estimates, we only use $$2\Re (P_1u,P_2u)\leq \left\|w\right\|^2_{L^2(-T,T[\times \Omega)}.$$ We also decompose $2\Re (P_1u,P_2u)=I_1+I_2+I_3$ with
\bna
I_1&:=&2\Re \iint (2s\nabla \varphi\cdot\nabla u +s(\Delta \varphi)u)(-i\ubar_t+\Delta \ubar +s^2|\nabla \varphi|^2\ubar))\\
I_2&:=&2\Re \iint is \varphi_t u(-i\ubar_t+\Delta\ubar)\\
I_3&:=&2\Re \iint is \varphi_t u(s^2|\nabla \varphi|^2\ubar)=0
\ena
We first deal with $I_1$.
\bna
I_1&=&2\Re \iint (2s\nabla \varphi\cdot\nabla u +s(\Delta \varphi)u)((\Delta \ubar +s^2|\nabla \varphi|^2\ubar)-2\Re \iint i(2s\nabla \varphi\cdot\nabla u +s(\Delta \varphi)u)\ubar_t\\
&=&I_1^1+I_1^2.
\ena
Set $J=\iint (\nabla \varphi \cdot \nabla u)\Delta \overline{u}=-\iint \nabla \overline{u}\cdot \nabla (\nabla \varphi \cdot \nabla u))$. We have
\bna
\nabla \overline{u}\cdot \nabla (\nabla \varphi \cdot \nabla u))&=&D_{\nabla \overline{u}}(\nabla \varphi \cdot \nabla u)=(D_{\nabla \overline{u}}\nabla \varphi )\cdot \nabla u+\nabla \varphi \cdot (D_{\nabla \overline{u}}\nabla u)\\
&=& Hess(\varphi)(\nabla u,\nabla\overline{u})+ Hess(u)(\nabla \ubar,\nabla \varphi)
\ena
Actually
\bna
\nabla \varphi \cdot \nabla |\nabla u|^2&=&D_{\nabla \varphi}(\nabla u\cdot \nabla \overline{u})=(D_{\nabla \varphi}\nabla u)\cdot \nabla \overline{u}+\nabla u\cdot (D_{\nabla \varphi}\nabla \overline{u})=2 \Re (D_{\nabla \varphi}\nabla u)\cdot \nabla \overline{u}\\
&=&2\Re Hess(u)(\nabla \varphi, \nabla \overline{u})
\ena
Therefore,
\bna
2\Re J=-2 \iint Hess(\varphi)(\nabla u,\nabla\overline{u})+\iint \Delta \varphi \left|\nabla u\right|^2
\ena
Expanding $I_1^1$, we obtain
\bna
I_1^1&=&2\Re \left\{ 2sJ+\iint s(\Delta \varphi)u\Delta \ubar +\iint 2s^3(\nabla \varphi\cdot\nabla u) |\nabla \varphi|^2\ubar+\iint s^3(\Delta \varphi)|u|^2|\nabla \varphi|^2\right\}\\
&=&4s\Re J -2s \Re \iint \left((\nabla \Delta \varphi )u+ \Delta \varphi \nabla u\right)\cdot \nabla \ubar\\
& &+\iint 2s^3|\nabla \varphi|^2 \nabla \varphi\cdot \nabla |u|^2) +2\iint s^3(\Delta \varphi)|u|^2|\nabla \varphi|^2
\ena
where we have used $\nabla |u|^2=2\Re (\ubar \nabla u )$. Then, we remark that
\bna
-2s \Re \iint (\nabla \Delta \varphi )u\cdot \nabla \ubar &=&-s  \iint (\nabla \Delta \varphi )\cdot \nabla |u|^2\\
&=&s  \iint (\Delta^2 \varphi ) |u|^2, \\
\ena
\bna
2\iint s^3(\Delta \varphi)|u|^2|\nabla \varphi|^2= -2s^3 \iint \nabla \varphi \cdot ( |\nabla \varphi|^2\nabla |u|^2 + |u|^2\nabla |\nabla \varphi|^2).
\ena
We simplify
\bna
I_1^1&=&-4s\Re \iint Hess(\varphi)(\nabla u,\nabla\overline{u})+2s\iint \Delta \varphi \left|\nabla u\right|^2\\
&&+s \iint (\Delta^2 \varphi ) |u|^2-2s \iint \Delta \varphi |\nabla u|^2 -2s^3 \iint  |u|^2\nabla \varphi \cdot \nabla |\nabla \varphi|^2\\
&=&-4s\iint Hess(\varphi)(\nabla u,\nabla\overline{u})+s \iint (\Delta^2 \varphi ) |u|^2-2s^3 \iint (\nabla \varphi \cdot \nabla |\nabla \varphi|^2)|u|^2 \ena
Expanding $2\Re a = a+\overline{a}$ for $I_1^2$ and performing integration by part in $t$ for the first term, we get 
\bna
-I_1^2&=& \iint i(2s\nabla \varphi\cdot\nabla u +s(\Delta \varphi)u)\ubar_t -i \iint (2s\nabla \varphi\cdot\nabla \ubar +s(\Delta \varphi)\ubar )u_t\\
&=& \iint -i\left[2s\nabla \varphi_t\cdot\nabla u +2s\nabla \varphi\cdot\nabla u_t+s(\Delta \varphi_t)u+s(\Delta \varphi)u_t\right]\ubar \\
&&-i \iint 2s(\nabla \varphi\cdot\nabla \ubar) u_t -i\iint s(\Delta \varphi)\ubar u_t
\ena
Integration by part in $x$ yields 
\bna
-i \iint 2s(\nabla \varphi\cdot\nabla \ubar) u_t= 2is\iint (\Delta \varphi) \ubar u_t + 2is \iint (\nabla\varphi \cdot \nabla u_t)\ubar 
\ena
As a consequence
\bna
-I_1^2&=& \iint -i2s(\nabla \varphi_t\cdot\nabla u )\ubar -i s\iint(\Delta \varphi_t)|u|^2= \iint -i2s(\nabla \varphi_t\cdot\nabla u )\ubar +i s\iint \nabla \varphi_t\cdot\nabla |u|^2\\
&=&i \iint s\nabla \varphi_t \cdot (u\nabla\ubar-\ubar \nabla u))= 2s\Re i\iint \nabla \varphi_t \cdot (u\nabla \ubar)).
\ena
Finally,
\bna
I_1&=&-4s\Re \iint Hess(\varphi)(\nabla u,\nabla\overline{u})+s \iint (\Delta^2 \varphi ) |u|^2 \\
&&-2s^3 \iint \nabla \varphi \cdot \nabla |\nabla \varphi|^2 |u|^2-2s\Re i\iint \nabla \varphi_t \cdot (u\nabla \ubar))
\ena
On the other hand, we have
$$ \nabla \varphi \cdot \nabla |\nabla \varphi|^2=D_{\nabla \varphi}(\nabla \varphi\cdot \nabla \varphi )=2 D_{\nabla \varphi}\nabla \varphi\cdot \nabla \varphi =2 Hess(\varphi)(\nabla \varphi,\nabla \varphi)$$
We now turn to the other term $I_2$ :
\bna
I_2&=&2\Re \iint is \varphi_t u(-i\overline{u}_t+\Delta \ubar)= s\iint \varphi_t \partial_t|u|^2 +2s\Re i \iint  \varphi_t u\Delta \ubar\\
&=& - s\iint \varphi_{tt} |u|^2 - 2s\Re i \iint  (\nabla\varphi_t u + \varphi_t \nabla u)\cdot\nabla  \ubar \\
&=& - s\iint \varphi_{tt} |u|^2 - 2s\Re  \iint  i(\nabla\varphi_t \cdot \nabla  \ubar) u \\
\ena
Consequently, our final result is 
\bnan
2\Re (M_1u,M_2u)&=&\label{partprincu} \iint \left[-4s^3Hess(\varphi)(\nabla \varphi,\nabla \varphi) - s \varphi_{tt}+s(\Delta^2 \varphi )\right]|u|^2 \\
\label{partprincgradu}&&-4s\Re \iint Hess(\varphi)(\nabla u,\nabla\overline{u}) \\
\label{reste}& &-4s\Re  \iint  iu\nabla\varphi_t \cdot \nabla  \ubar 
\enan
(\ref{partprincu}) and (\ref{partprincgradu}) are the main parts in $|u|^2$ and $|\nabla u|^2$ respectively. (\ref{reste}) is a remainder term that will be estimated from above.

In what follows, $\varepsilon>0$ denote small constants (used in estimates from below) and $C$ large ones (used for estimates from above). 
We observe the following indentities, that will be used along the proof,
\bna
\nabla \varphi =- \lambda \theta \nabla\Psi, 
\ena
\bna
Hess(\varphi)(X,Y)&=&(D_X \nabla \varphi) \cdot Y=-\lambda D_X (\theta \nabla \Psi) \cdot Y=-\lambda \theta (D_X \nabla \Psi) \cdot Y -\lambda d\theta(X)\nabla \Psi \cdot Y\\
&=&-\lambda \theta Hess(\Psi)(X,Y)-\lambda^2\theta(\nabla \Psi \cdot X)(\nabla \Psi \cdot Y)\\
&=&-\theta \lambda \left[Hess(\Psi)(X,Y)+\lambda(\nabla \Psi \cdot X)(\nabla \Psi \cdot Y)\right].
\ena
Firstly, we estimate term (\ref{reste}), 
\bnan
\label{ineginterm66}
\left|(\ref{reste})\right|&\leq & Cs \iint |\nabla \varphi _t  \cdot \nabla u| |u|\leq C s \iint \frac{t\lambda e^{\lambda \Psi}}{(T^2-t^2)^2}|\nabla \Psi  \cdot \nabla u| |u| \nonumber\\
&\leq & C s \iint \frac{e^{\lambda \Psi}}{(T^2-t^2)}|\nabla \Psi  \cdot \nabla u|^2 +C s \iint \frac{(T\lambda)^2 e^{\lambda \Psi}}{(T^2-t^2)^3}|u|^2\nonumber\\
&\leq &Cs \iint \theta \left|\nabla \Psi \cdot \nabla u\right|^2+Cs\lambda^{-1}\iint \left|\nabla \varphi\right|^3\left|u\right|^2+C s \iint_{\omega} \lambda^2 \theta^3|u|^2 
\enan
Then, we estimate term (\ref{partprincu}) using  Assumptions (\ref{gradnonnul}) and (\ref{weakconvex}) (or (\ref{strongconvex})). On $(\Omega\backslash \omega) \cap K $, we have
\bna
-4s^3Hess(\varphi)(\nabla \varphi,\nabla \varphi)&=&4s^3\lambda \theta  \left[Hess(\Psi)(\nabla\varphi,\nabla\varphi)+\lambda\left|\nabla \Psi \cdot \nabla \varphi\right|^2\right]\\
&\geq &4s^3\lambda \theta  (\lambda-1) \left|\nabla \Psi \cdot \nabla \varphi \right|^2\geq s^3\lambda^4 \theta^3\left|\nabla \Psi\right|^4\geq \varepsilon s^3\lambda \left|\nabla \varphi\right|^3
\ena
Assumption (\ref{assumCpsi}) gives $\Psi(x)\leq C_{\Psi}\leq 3\Psi(x)$ and then, we have on $(\Omega\backslash \omega)\cap K$
\bna
\left|s\varphi_{tt}\right|\leq Cs\frac{e^{\lambda C_{\Psi}}}{((T^2-t^2))^3}\leq Cs\frac{e^{3\lambda \Psi(x)}}{((T^2-t^2))^3}\leq Cs\left|\nabla \varphi\right|^3 
\ena
Moreover, on $(\Omega\backslash \omega)\cap K $ we have 
\bna
\left|s\Delta^2\varphi\right|\leq Cs \theta \lambda^4 \leq Cs \lambda \left|\nabla \varphi \right|^3
\ena
Finally, for $\lambda$ and $s$ large enough
\bna
\iint_{\Omega\backslash \omega} \left[-4s^3Hess(\varphi)(\nabla \varphi,\nabla \varphi) - s \varphi_{tt}+s(\Delta^2 \varphi )\right]|u|^2\geq \iint_{\Omega\backslash \omega} \varepsilon s^3\lambda \left|\nabla \varphi\right|^3|u|^2
\ena
For the domain $\omega$, we have the estimate
\bna
\left|\iint_{\omega} \left[-4s^3Hess(\varphi)(\nabla \varphi,\nabla \varphi) - s \varphi_{tt}+s(\Delta^2 \varphi )\right]|u|^2\right|\leq C\iint_{\omega} s^3\lambda^4 \theta^3|u|^2
\ena
The final estimate for (\ref{partprincu}) is
\bnan
\label{ineginterm64}
(\ref{partprincu}) \geq \iint_{\Omega\backslash \omega} \varepsilon s^3\lambda \left|\nabla \varphi\right|^3|u|^2 -C \iint_{\omega} s^3\lambda^4 \theta^3|u|^2.
\enan
Now, let us estimate (\ref{partprincgradu}). We begin with the integral on $\omega$.
\bna
-4s\Re \iint_{\omega} Hess(\varphi)(\nabla u,\nabla\overline{u})&=&4s\Re \iint_{\omega} \theta \lambda \left[Hess(\Psi)(\nabla u,\nabla \overline{u})+\lambda\left|\nabla \Psi \cdot \nabla u \right|^2\right] \\
&\geq & - Cs\lambda\iint_{\omega}\theta \left|\nabla u \right|^2 +4 s \iint_{\omega} \theta \lambda ^2\left|\nabla \Psi \cdot \nabla u \right|^2\\
&\geq & - Cs\lambda\iint_{\omega}\theta \left|\nabla u \right|^2
\ena
Now, for the integral on $\Omega\backslash \omega$, we distinguish the two cases described above :

\bigskip

\textbf{Strong pseudoconvexity} : end of the proof of Proposition \ref{propCarlemanstrong}\\
Using assumption (\ref{strongconvex}), we can estimate the  part of (\ref{partprincgradu}) on $\Omega\backslash \omega$ by 
\bna
-4s\Re \iint_{\Omega\backslash \omega} Hess(\varphi)(\nabla u,\nabla\overline{u})&=&4s\Re \iint_{\Omega\backslash \omega}\theta \lambda \left[Hess(\Psi)(\nabla u,\nabla \overline{u})+\lambda\left|\nabla \Psi \cdot \nabla u \right|^2\right] \\
&\geq &\varepsilon s\lambda  \iint_{\Omega\backslash \omega} \theta \left|\nabla u\right|^2
\ena
The final estimate for (\ref{partprincgradu}) is 
\bnan
\label{ineginterm65}
(\ref{partprincgradu}) \geq \varepsilon s\lambda  \iint_{\Omega\backslash \omega} \theta \left|\nabla u\right|^2- Cs\lambda\iint_{\omega}\theta \left|\nabla u \right|^2
\enan
Putting together (\ref{ineginterm66}), (\ref{ineginterm64}) and (\ref{ineginterm65}), we get for $s$, $\lambda$ large enough
\bnan
(\ref{partprincu})+(\ref{partprincgradu})+(\ref{reste})&\geq &\iint_{\Omega\backslash \omega} \varepsilon s^3\lambda \left|\nabla \varphi\right|^3|u|^2 - C \iint_{\omega} s^3\lambda^4 \theta^3|u|^2- C s\lambda\iint_{\omega}\theta \left|\nabla u \right|^2\nonumber\\
&&+\varepsilon s\lambda  \iint_{\Omega\backslash \omega}\theta \left|\nabla u\right|^2-Cs\iint \theta \left|\nabla \Psi \cdot \nabla u\right|^2\nonumber\\
&&-C s\lambda^{-1}\iint \left|\nabla \varphi\right|^3\left|u\right|^2-C s \iint_{\omega} \lambda^2 \theta^3|u|^2\nonumber\\
&\geq & \varepsilon \iint s^3\lambda^4 \theta^3 |u|^2+\varepsilon s\lambda  \iint \theta \left|\nabla u\right|^2\nonumber\\
&&- C \iint_{\omega} s^3\lambda^4 \theta^3|u|^2- C s\lambda\iint_{\omega}\theta \left|\nabla u \label{presquefinal} \right|^2
\enan
where we have used the decomposition $\iint_{\Omega\backslash \omega}=\iint -\iint_{\omega}$ for the second inequality.\\
Replacing $u$ by $e^{-s\varphi}q$ and computing $\nabla q=e^{s\varphi}\left[\nabla u-s \lambda\theta  u \nabla \Psi\right]$ this yields after absorption 
\bnan
\iint \left[s^3\lambda^4 \theta^3 |q|^2+s\lambda  \theta \left|\nabla q\right|^2\right]e^{-2s\varphi} & \leq & C \iint \left[s^3\lambda^4 \theta^3 |u|^2 +s\lambda \theta |\nabla u|^2+s^3\lambda^3 \theta^3|\nabla \psi|^2|u|^2\right]\nonumber\\
&\leq &C \iint \left[s^3\lambda^4 \theta^3 |u|^2 +s\lambda \theta |\nabla u|^2\right]\label{lienuq}
\enan
\bnan
 \iint_{\omega} s^3\lambda^4 \theta^3|u|^2+ s\lambda\iint_{\omega}\theta \left|\nabla u \right|^2 & \leq &C\iint_{\omega} \left[s^3\lambda^4 \theta^3|q|^2+s\lambda\theta \left|\nabla q \right|^2+s^3\lambda^3 \theta^3|\nabla \psi|^2|q|^2\right]e^{-2s\varphi}\nonumber\\
& \leq &C\iint_{\omega} \left[s^3\lambda^4 \theta^3|q|^2+s\lambda\theta \left|\nabla q \right|^2\right]e^{-2s\varphi} \label{lienuq2}
\enan
Combining (\ref{presquefinal}), (\ref{lienuq}) and (\ref{lienuq2}), we get the expected result :
\bna
&&\iint \left[s^3\lambda^4 \theta^3 |q|^2+s\lambda  \theta \left|\nabla q\right|^2\right]e^{-2s\varphi}\\
&\leq & C\iint\left|i\partial_tq+\Delta q\right|^2e^{-2s\varphi}+C\iint_{\omega} \left[s^3\lambda^4 \theta^3|q|^2+s\lambda\theta \left|\nabla q \right|^2\right]e^{-2s\varphi}\nonumber
\ena

\bigskip

\textbf{Weak pseudoconvexity} : end of the proof of Proposition \ref{propCarlemanweak}\\
Assumption (\ref{weakconvex}) yields that for $\lambda$ large enough
 \bna
-4s\Re \iint_{\Omega\backslash \omega} Hess(\varphi)(\nabla u,\nabla\overline{u})\geq  \varepsilon s \iint_{\Omega\backslash \omega}\theta \lambda^2 \left|\nabla \Psi \cdot \nabla u \right|^2
\ena 
We finish the proof similarly to get 
\bna
(\ref{partprincu})+(\ref{partprincgradu})+(\ref{reste})
&\geq &\varepsilon\iint_{\Omega\backslash \omega}  s^3\lambda \left|\nabla \varphi\right|^3|u|^2 - C \iint_{\omega} s^3\lambda^4 \theta^3|u|^2- C s\lambda\iint_{\omega}\theta \left|\nabla u \right|^2\nonumber\\
&&+\varepsilon s \iint_{\Omega\backslash \omega}\theta \lambda^2 \left|\nabla \Psi \cdot \nabla u \right|^2-Cs\iint \theta \left|\nabla \Psi \cdot \nabla u\right|^2\nonumber\\
&&-C s\lambda^{-1}\iint \left|\nabla \varphi\right|^3\left|u\right|^2-C s \iint_{\omega} \lambda^2 \theta^3|u|^2\nonumber\\
&\geq & \varepsilon \iint s^3\lambda^4 \theta^3 |u|^2+\varepsilon s\lambda ^2 \iint \theta\left|\nabla\Psi\cdot\nabla u\right|^2\\
&&- C\iint_{\omega} s^3\lambda^4 \theta^3|u|^2- Cs\lambda\iint_{\omega}\theta \left|\nabla u \right|^2
\ena
and then
\bna
&&\iint \left[s^3\lambda^4 \theta^3 |q|^2+s\lambda^2  \theta \left|\nabla \Psi\cdot\nabla q \right|^2\right]e^{-2s\varphi}\\
&\leq & C\iint\left|i\partial_tq+\Delta q\right|^2e^{-2s\varphi}+C\iint_{\omega} \left[s^3\lambda^4 \theta^3|q|^2+s\lambda\theta \left|\nabla q \right|^2\right]e^{-2s\varphi}.\nonumber
\ena 
\end{proof}
\subsection{Carleman estimates with potential $L^{\infty}([-T,T],L^3)$}
\label{sectionuniquepot}
The following result proves that the strong pseudoconvexity allows to absorb some potential terms in $L^{\infty}([-T,T],L^3)$. This is in contrast with the weak pseudoconvexity which only absorbs terms in $L^{\infty}([-T,T]\times \Omega)$.
\begin{prop}
Assume $dim(\Omega)\leq 3$. Let $V_1$, $V_2\in L^{\infty}([-T,T],L^3)$. Then, Proposition \ref{propCarlemanstrong} holds with $L$ replaced by
$$L(q)=i\partial_t q +\Delta q +V_1q+V_2 \overline{q}. $$
\end{prop}
\begin{proof}
We use the notation of Proposition \ref{propCarlemanstrong}. We write
\bna
\iint \left|i\partial_t q+\Delta q\right|^2e^{-2s\varphi} &\leq& 4\nor{e^{-s\varphi}L(q)}{L^2([0,T],L^2)}^2 + 4\nor{e^{-s\varphi}(V_1 q)}{L^2([0,T],L^2)}^2+4\nor{e^{-s\varphi}(V_2 \overline{q})}{L^2([0,T],L^2)}^2
\ena
But, by Hölder inequality and Sobolev embedding, we have for $s>1$
\bna
\nor{e^{-s\varphi}V_1 q}{L^2([0,T],L^2)}^2 &\leq &C\nor{V_1}{L^{\infty}(L^3)}^2\nor{e^{-s\varphi}q}{L^2(L^6)}^2\\
& \leq & C\left(\nor{e^{-s\varphi}q}{L^2(L^2)}^2+\nor{\nabla (e^{-s\varphi}q)}{L^2(L^2)}^2\right)\\
&\leq& C \left(\nor{e^{-s\varphi}q}{L^2(L^2)}^2+\nor{e^{-s\varphi}\nabla q}{L^2(L^2)}^2+s^2\lambda^2\nor{\theta(\nabla\Psi) e^{-s\varphi} q}{L^2(L^2)}^2\right)\\
&\leq& C \left(\iint \left[s^2\lambda^2 \theta^3 |q|^2+\theta \left|\nabla q\right|^2\right]e^{-2s\varphi}\right)
\ena
where we have used $\theta\geq C$. We get the desired result using estimate (\ref{Carlemanstrong}) of Proposition \ref{propCarlemanstrong} for $s$ large enough.
\end{proof}
\begin{remarque}
The uniqueness results we will obtain from the former Proposition are not optimal with respect to the regularity of the potential. Indeed, some recent papers (see the work of H. Koch and D. Tataru \cite{KochTataruCarlLp} or D. Dos Santos Ferreira \cite{DDSFLp}) establish Carleman type estimates in $L^p$ which are much better than what we get. They are more complicated and not required for our purpose. Yet, they would become necessary if we considered nonlinearities $|u|^{\alpha}u$ with $\alpha>2$.
\end{remarque}
\subsection{Application to uniqueness}
\begin{prop}
\label{thmprolongunique}
Let $\Omega, T, \omega, \Psi$ fulfilling the same assumptions as Proposition \ref{propCarlemanstrong}.\\
Let $q\in L^{\infty}([-T,T],H^1(\Omega))$ compactly supported, solution of $i\partial_t q +\Delta q +V_1q+V_2\overline{q}=0$ with $V_i\in L^{\infty}([-T,T],L^3)$ .\\
Let $D$ be an open subset of $\Omega$ such that $\widetilde{m}=\inf_{x\in D}\left\{\Psi(x)\right\}> \sup_{x\in \omega} \left\{\Psi(x)\right\}=m$.\\
Then, $q=0$ on $]-T,T[\times D$.
\end{prop}
\begin{remarque}
By considering the maximum of $\Psi$, we see that the assumptions of Proposition \ref{thmprolongunique} can not be fulfilled on a compact manifold. Therefore, we will only apply this result on an open set $\Omega$ of $M$, and the compact support of $u$ becomes important.
\end{remarque}
Since the previous Carleman estimates hold for every time interval (with constants depending on its length), we are reduced to the following lemma : 
\begin{lemme}
Under assumptions of Proposition \ref{thmprolongunique}, there exists one $\eta>0$ such that $q=0$ on $]-\eta,\eta[\times D$.
\end{lemme}
\begin{proof}
Fix $\lambda\geq \lambda_0>1$ (the next constants could depend on $\lambda$ but not on $s$). Let $T\geq \eta>0$ to be chosen later. Denote $\lambda_1=e^{\lambda C_{\psi}}-e^{\lambda \widetilde{m}}$ and $\lambda_1+\varepsilon=e^{\lambda C_{\psi}}-e^{\lambda m}$ with $\lambda_1>0$ and $\varepsilon>0$. By definition of $\widetilde{m}$ and $m$, we have for $s\geq 0$
\bna
e^{-2s\varphi} \leq e^{-2s\frac{\lambda_1+\varepsilon}{T^2-t^2}} \quad \forall (t,x)\in ]-T,T[\times \omega\\
e^{-2s\frac{\lambda_1}{T^2-\eta^2}} \leq e^{-2s\varphi} \quad \forall (t,x)\in ]-\eta,\eta[\times D
\ena
Moreover, once $\lambda_1$ and $\varepsilon$ are fixed, there exists some constant $C$ such that $y^3e^{-2(\lambda_1+\varepsilon)y}\leq Ce^{-2(\lambda_1+\varepsilon/2)y}$ for $y\geq 0$. Therefore, for every $(t,x)\in ]-T,T[\times \Omega$ with $x\in Supp ~u$, we have
$$(s\theta)^3 e^{-2s\frac{\lambda_1+\varepsilon}{T^2-t^2}}\leq C\left(\frac{s}{T^2-t^2}\right)^3 e^{-2s\frac{\lambda_1+\varepsilon}{T^2-t^2}}\leq C e^{-2s\frac{\lambda_1+\varepsilon/2}{T^2-t^2}}\leq C e^{-2s\frac{\lambda_1+\varepsilon/2}{T^2}}$$
Here, the constant $C$ does not depend on $s$.
Then, using Carleman estimate and $\theta \geq C>0$, we get
\bna
&&\iint_{]-\eta,\eta[\times D}s^3|q|^2e^{-2s\frac{\lambda_1}{T^2-\eta^2}}
\leq  C\iint_{]-T,T[\times \omega} \left[|u|^2+\left|\nabla q \right|^2\right]e^{-2s\frac{\lambda_1+\varepsilon/2}{T^2}}
\ena
Therefore,
\bna
s^3e^{-2s\frac{\lambda_1}{T^2-\eta^2}} \iint_{]-\eta,\eta[\times D}|q|^2
\leq  Ce^{-2s\frac{\lambda_1+\varepsilon/2}{T^2}} \nor{q}{L^{2}(H^1)}^2
\ena
Then, to finish the proof, we just have to choose $\eta$ such that $-2\frac{\lambda_1}{T^2-\eta^2} > -2\frac{\lambda_1+\varepsilon/2}{T^2}$, that is $\eta^2<\frac{T^2\varepsilon/2}{\lambda_1+\varepsilon/2}$ and let $s$ tend to $+\infty$.
\end{proof}
\subsection{Geometrical examples}
We give some geometrical examples where Proposition \ref{thmprolongunique} applies. Denote $q\in L^{\infty}([-T,T],H^1(\Omega))$ a solution of $i\partial_t q +\Delta q +V_1q+V_2\overline{q}=0$ with $V_i\in L^{\infty}([-T,T],L^3)$. In these following cases, Assumptions \ref{uniquecontinuation} and \ref{uniquecontinuationH1} are fulfilled. For the convenience of the reader, we recall this assumption : 
\begin{prop}
Let $(M,\widetilde{\omega})$ be either\\
- $(\Tot,\left\{x\in \R^3/(\theta_1 \Z\times \theta_2 \Z\times \theta_3\Z)  \left|\exists i\in \{1,2,3\}, x_i\in ]-\varepsilon,\varepsilon[+\theta_i\Z \right.\right\})$\\
- $(S^3,\widetilde{\omega})$ where $\widetilde{\omega}$ is a neighborhood of $S^3\cap\left\{x_4=0\right\}$ in $S^3\subset \R^4$.\\
-$(S^2\times S^1,(\omega_1 \times S^1 )\cup (S^2 \times ]0,\varepsilon[))$ where $\omega_1$ is a neighborhood of the equator of $S^2$.\\
For every $T>0$, the only solution in $C([0,T],H^1)$ to the system 
\begin{eqnarray}
\left\lbrace
\begin{array}{c}
i\partial_t q + \Delta q + b_1(t,x)q+b_2(t,x)\overline{q}=0  \textnormal{ on } [0,T]\times M\\
q=0 \textnormal{ on } [0,T]\times \widetilde{\omega}
\end{array}
\right.
\end{eqnarray}
where $b_1(t,x)$ and $b_2(t,x) \in L^{\infty}([0,T],L^{3})$ is the trivial one $q \equiv 0$.
\end{prop}
\subsubsection{$M=\Tot$}
We assume $q=0$ on $\widetilde{\omega}=\left\{x\in \R^3/(\theta_1 \Z\times \theta_2 \Z\times \theta_3\Z)  \left|\exists i\in \{1,2,3\}, x_i\in ]-\varepsilon,\varepsilon[+\theta_i\Z \right.\right\}$.\\
We define $\widetilde{q}$ on $\R^3$ by $\widetilde{q}(x)=q(x) $ if $x\in [0,\theta_1]\times[0,\theta_2]\times[0,\theta_3]$ and $\widetilde{q}(x)=0$ otherwise. $\widetilde{q}$ satisfies the same Schrödinger equation on $\R^3$ with compact support $K$. By translation, we can assume that $0$ is the center of the rectangle.\\
We use the function $\Psi=\left\|(x,y,z)\right\|^2+C$. $C$ is chosen large enough so that (\ref{assumCpsi}) is fulfilled on $K$. Let $\delta>0$ small. Outside of $\omega=B(0,\delta)$, $\Psi$ is stricly convex ( that is strongly pseudoconvex for the flat metric inherited from $\R^3$) and $\nabla\Psi \neq 0$. Then, assumptions (\ref{gradnonnul}) and (\ref{strongconvex}) are fulfilled.\\
We can apply Theorem \ref{thmprolongunique} with $\Omega = \R^3$, $\omega=B(0,\delta)$ and $D=B(0,2\delta)^c$. As $\delta$ is arbitrary, we get $\widetilde{q}=0$ everywhere and so $q=0$.
\subsubsection{$M=S^3$}
\begin{lemme}
\label{convexSn}
Let $S^n\subset \R^{n+1}$ be the unit sphere. Then, the function $h : (x_1,\troisp,x_{n+1})\mapsto x_{n+1}$ restricted to $S^n\cap \left\{x_{n+1}<0\right\}$ has stricly positive Hessian for the metric induced by $\R^{n+1}$. 
\end{lemme}
\begin{proof}
$h$ defined on $\R^{n+1}$ is linear. Then, using Exercice 2.65 b) of \cite{GallotHulinLaf}, we get $Hess(h)= - h g$ where $g$ is the bilinear form of the Riemannian structure. Then, $Hess(h)$ is positive definite if and only if $h<0$.
\end{proof}
We assume $q=0$ on a neighborhood of $x_4=0$. Let $\delta>0$ small. We choose $\Omega=\{x\in S^3|x_4<0 \}$, $D=\{x\in S^3|x_4\in ]-1+2\delta,0[ \}$ and $\omega=S^3\cap \{x_4\in [-1,-1+\delta[ \}$. We use the function $\Psi=x_4+C$. $C$ is chosen large enough so that (\ref{assumCpsi}) is fulfilled on the support of $q$. On $\Omega \backslash ~\omega$, $\Psi$ is stricly convex thanks to Lemma \ref{convexSn} and $\nabla\Psi \neq 0$. Therefore, assumptions (\ref{gradnonnul}) and (\ref{strongconvex}) are fulfilled. As the support of $q$ is compact in $\Omega$, Theorem \ref{thmprolongunique} applies and we get $q=0$ on $D$. Since $\delta$ is arbitrary, we get $q=0$ on $S^3\cap \{x_4<0 \}$. The symmetry of the problem gives $q=0$ on $S^3$.
\subsubsection{$M=S^2\times S^1$}
Let $\omega_1\subset S^2$ be a neighbourhood of the equator $\left\{x_3=0\right\}$ and $\varepsilon>0$.\\
We assume $q=0$ on $\left(\omega_1\times S^1\right)\bigcup \left(S^2\times ]-\varepsilon,\varepsilon[\right)$.\\
The geometric situation is quite similar to the case of $\Tot$ : this is a product of manifolds and the weight function $\Psi$ will be the sum of two pseudoconvex weights in each coordinate.\\
The current point $x$ of $S^2$ will be denoted by its coordinates in $\R^3$ and the current point $y$ of $S^1=\Tu=\R/\Z$ by its coordinates in $\R$. Then, we can define $\widetilde{q}$ on the open set $\Omega=\{x\in S^2|x_3<0 \} \times \R$ by $\widetilde{q}(x,y)=q(x,y)$ if $y\in [0,1]$ and $0$ otherwise. $\widetilde{q}$ is then compactly supported and is solution of the same Schrödinger equation.\\
 We choose $\Psi(x,y)=x_3+y^2+C$ with $C$ large enough. $\Psi$ is definite positive everywhere and nonsingular everywhere outside of any $\omega= \{(x,y)\in S^2\times \R|x_3\in [-1,-1+\delta[ \textnormal{ and } y^2<\delta\}$ for $\delta>0$. Then, choosing $D= \left\{(x,y)\in S^2\times \R\left|x_3\in ]-1+3\delta,0[ \textnormal{ or } y^2>3\delta\right.\right\}$ and applying Theorem \ref{thmprolongunique} we get $\widetilde{q}=0$ on $D$. Therefore, $q=0$ on $S^2\times S^1$.
 
 \bigskip
 
{\it Acknowledgements.} The author deeply thanks his adviser Patrick G\'erard for attracting his attention to this problem and for helpful discussions and encouragements. He also thanks Belhassen Dehman for some explanations about his article \cite{HumDehLeb} and David Dos Santos Ferreira for enlightenments about unique continuation in low regularity.
\bibliographystyle{plain} 
\bibliography{biblio}
\end{document}